\PassOptionsToPackage{svgnames}{xcolor}

\documentclass[preprint]{elsarticle}

\usepackage[margin=0.7in]{geometry}

\usepackage{lineno}
\usepackage[allcolors=black,hidelinks = true]{hyperref}
\modulolinenumbers[1]

\journal{Linear Algebra and its Applications}

\usepackage[english]{babel}
\usepackage{graphicx,epstopdf,epsfig}
\usepackage{algorithmic}
\usepackage{amsfonts,epsfig,fancyhdr,graphics, hyperref,amsmath,amssymb}
\usepackage[toc,page]{appendix}
\usepackage{mathtools}

\ifpdf
  \DeclareGraphicsExtensions{.eps,.pdf,.png,.jpg}
\else
  \DeclareGraphicsExtensions{.eps}
\fi
\usepackage[font=footnotesize]{caption}

\usepackage{multirow}
\usepackage{todonotes}
\usepackage{booktabs}
\usepackage{varwidth}

\usepackage{amsthm}

\newtheorem{Theorem}{Theorem}
\newtheorem{Lemma}{Lemma}

\newtheorem{Definition}{Definition}
\newtheorem{Remark}{Remark}

\newproof{pf}{Proof}

\newcommand{\N}{\mathbb{N}}

\newcommand{\C}{\mathbb{C}}

\newcommand{\E}{{\rm e}}

\usepackage{tikz,pgfplots}
\usetikzlibrary{decorations.markings}
\usetikzlibrary{shapes.geometric}
\pgfdeclarelayer{edgelayer}
\pgfdeclarelayer{nodelayer}
\pgfsetlayers{edgelayer,nodelayer,main}
\tikzstyle{none}=[inner sep=0pt]
\tikzstyle{wh}=[circle,fill=White,draw=Black,line width=0.8 pt]
\tikzstyle{rn}=[circle,fill=Red,draw=Black,line width=0.8 pt]
\tikzstyle{gn}=[circle,fill=Lime,draw=Black,line width=0.8 pt]
\tikzstyle{yn}=[circle,fill=Yellow,draw=Black,line width=0.8 pt]
\tikzstyle{simple}=[-,draw=Black,line width=2.000]
\tikzstyle{arrow}=[-,draw=Black,postaction={decorate},decoration={markings,mark=at position .5 with {\arrow{>}}},line width=2.000]
\tikzstyle{tick}=[-,draw=Black,postaction={decorate},decoration={markings,mark=at position .5 with {\draw (0,-0.1) -- (0,0.1);}},line width=2.000]

\usepackage{subfig}
\usepackage{booktabs}
\hypersetup{breaklinks=true}

\renewcommand{\v}[1]{{\mathbf #1}}

\DeclareMathOperator*{\diag}{diag}

\begin{document}
\begin{frontmatter}

\title{Symbol based convergence analysis in multigrid methods for saddle point problems}

\author[1]{Matthias Bolten}

\author[2]{Marco Donatelli}

\author[2]{Paola Ferrari}

\author[1]{Isabella Furci*}
\cortext[mycorrespondingauthor]{Corresponding author}
\ead{furci@uni-wuppertal.de}

\address[1]{{{School of Mathematics and Natural Sciences}}, {University of Wuppertal}, {{Wuppertal}, {Germany}}}
\address[2]{{Department of Science and high Technology}, {University of Insubria}, {{Como}, {Italy}}}

\begin{abstract}
Saddle point problems arise in a variety of applications, e.g., when solving the Stokes equations. They can be formulated such that the system matrix is symmetric, but indefinite, so the variational convergence theory that is usually used to prove multigrid convergence cannot be applied. In a 2016 paper in Numerische Mathematik Notay has presented a different algebraic approach that analyzes properly preconditioned saddle point problems, proving convergence of the Two-Grid method.

In the present paper we analyze saddle point problems where the blocks are circulant within this framework. We are able to derive sufficient conditions for convergence and provide optimal parameters for the preconditioning of the saddle point problem and for the point smoother that is used. The analysis is based on the generating symbols of the circulant blocks. Further, we show that the structure can be kept on the coarse level, allowing for a recursive application of the approach in a W- or V-cycle and proving the ``level independency'' property. Numerical results demonstrate the efficiency of the proposed method in the circulant and the Toeplitz case.
\end{abstract}


\end{frontmatter}

\section{Introduction}\label{sec:introduction}
Saddle point linear systems arise in different cases. One of the most important examples is the discretization of the Stokes equations that are given by
\begin{align*}
\xi \mathbf{u} - \nu \Delta \mathbf{u} + \nabla p & = \mathbf{f}, \quad \text{in $\Omega$},\\
\nabla \cdot \mathbf{u} & = 0, \quad \text{in $\Omega$},
\end{align*}
where $\Omega \subset \mathbb{R}^d$, $d = 2,3$ and suitable boundary conditions are imposed. Here, $\mathbf{u}$ represents the velocity and $p$ the pressure. They give rise to linear systems  
\begin{equation}\label{eq:linear_system}
\mathcal{A} x = b,
\end{equation}
with
\begin{equation}\label{eq:matrix_saddle}
  \mathcal{A} = \begin{bmatrix}
    A & B^T \\
    B & -C
  \end{bmatrix},
\end{equation}
where $A$ {is symmetric positive definite,} $C$ {is symmetric nonnegative definite} and $B$ has full rank. The iterative solution of saddle point problems has been studied extensively, for an introductory overview we refer to \cite{MR2168342}. Here we focus on multigrid methods. Different approaches to solve saddle point problems using multigrid exist. In most cases more powerful smoothers are used to take into account the special coupling, represented by the off-diagonal blocks in \eqref{eq:matrix_saddle}. This includes the Braess-Sarazin smoother \cite{MR1438078}, the Uzawa smoother \cite{MR833993} and the Vanka smoother \cite{MR848451}, where the latter {one} is probably the most widely used. Usually, these smoothers are applied in geometric multigrid methods. As such they are often analyzed using local Fourier analysis (LFA). An introduction to LFA can be found in \cite{MR2108045} and the analysis of block smoothers like the ones mentioned, e.g., in \cite{MR3795547,MR2840198,MR3488076}.

Instead of altering the smoother Notay has recently presented a method that applies a point-wise smoother and a coarse grid correction to \eqref{eq:linear_system} after preconditioning from the left and the right using lower and upper triangular matrices, respectively \cite{MR3439215}.

In this paper we consider the case where the matrices $A, B, C$ in \eqref{eq:matrix_saddle} are $n \times n$ circulant matrices. Multigrid for circulant matrices are well-understood and have been analyzed in \cite{Serra_Possio} and the optimality of the V-cycle has been shown in \cite{AD}. Its relation to LFA is described in \cite{D10}. Systems of PDEs, like the Stokes equations, yield block matrices. Multigrid for Toeplitz matrices with small blocks has been studied in \cite{HS2} and recently in \cite{block_multigrid2,DFFSS}. These analyses share that they are based on a variational principle, that cannot be applied here, as the considered systems do not induce a scalar product. The analysis of numerical methods for structured matrices not only is of interest when the problem to solve is posed in a rectangular domain and possesses constant coefficients. The results also carry over to the case of non-constant coefficients by means of generalized locally Toeplitz (GLT) sequences \cite{glt1,glt2}, and the developed methods can also be used on more complex domains, e.g., by using fictitious domain techniques or suitable discretizations as in \cite{MR4207181,art:BOLT21a,inp:BOLT19}.

Analyzing the system matrix in the case of circulant blocks $A, B, C$ within the framework presented in \cite{MR3439215}  we are able to derive sufficient conditions based on the symbols of the matrices such that the requirements presented there are fulfilled. The symbol-based analysis allows to choose optimal parameters for the left and right preconditioners defined in \cite{MR3439215} and for the {damped} Jacobi methods used as smoothing procedure. Further, the analysis motivates the choice of the projector in the case of $C$ approaching the zero matrix. For the multigrid case, we propose a strategy that keeps the same structure on the coarse level, allowing for a recursive application in a W- or V-cycle. For this strategy we are able to show that the degree of the polynomial that represents the generating symbol is bounded, i.e., the bandwidth of the matrices on the coarse levels is bounded, as well. Moreover, we prove the ``level independency'' property, i.e., the two-grid optimality at a generic level of the multigrid method, which ensures a robust W-cycle method. Finally, numerical tests demonstrate the efficiency of the proposed method and the validity of the theoretical analysis.

This paper is organized as follows. Section \ref{sec:preliminary} defines the notation used in the paper, in particular concerning the symbol of circulant matrices. Section \ref{sec:multigrid} is devoted to recall the main results on the convergence of multigrid methods both for circulant matrices and for the Stokes problem. In particular, Section \ref{ssec:tgm_stokes} contains the main results on the convergence of the Two-Grid method (TGM) for matrices of the form \eqref{eq:matrix_saddle}. Section \ref{sec:TGM_conv_circ} is dedicated to the theoretical analysis of the TGM convergence in terms of the generating functions and in Section \ref{sec:numerical_TGM} we present a key example showing the numerical efficiency of the derived convergence results. Section \ref{sec:procedure} extends the TGM convergence analysis to multigrid methods providing a strategy to preserve the same structure of the coefficient matrices at the coarser levels and proving the ``level independency'' property.
The numerical results in Section \ref{sec:numerical_MGM} confirm the linear convergence of the W-cycle, while in Section \ref{sect:toep} the proposed multigrid method is applied to Toeplitz matrices obtaining the same optimal convergence behavior also with the V-cycle method.
Some conclusions and future research lines are drawn in Section \ref{sect:concl}.

\section{Notation and definitions} \label{sec:preliminary}
For $X \in \C^{n \times}$, we denote by $\Lambda(X)$ the set of all the eigenvalues of $X$ and with $\rho(X)$ its spectral radius.
If  $X$ and $Y$ are Hermitian matrices, then the notation $X \le Y$ (resp. $X < Y$ ) means that $Y - X$ is a nonnegative definite (resp. positive definite) matrix. Moreover, we numerate such eigenvalues adopting the following notation
\begin{equation*}
\lambda_{\min}(X)=\lambda_1\le\lambda_2\le \dots \le \lambda_n=\lambda_{\max}(X).
\end{equation*}
If $X$ is a Hermitian positive definite (HPD) matrix, then
$\|{\v{v}}\|_{X}={\|}X^{1/2}{\v{v}}{\|}_{2}$ {(resp. $\|Y\|_{X}=\|X^{1/2}YX^{-1/2}\|_{2}$)} denotes the Euclidean norm
weighted by $X$ on $\mathbb{C}^{n}$ { (resp. on $\mathbb{C}^{n\times n}$)}. 

We denote by $I_n$ and $O_n$ the $n\times n$ identity matrix and the matrix of all zeros respectively. Moreover,  $\v{e}_{n}$ and $\v{0}_n$ are respectively the  vectors of length $n$ of all ones and zeros. When the dimension is clear from the context, we omit the subscript $n$. Given a matrix $X$, we denote {by} $D_X=\diag(X)$ the diagonal matrix having {ones} as elements on the main diagonal of $X$. 

Defining the $n$ equispaced grid points  
\[\theta_{j,n}=\frac{(j-1)2\pi}{n},\qquad j=1,\ldots,n,\] 
for the interval $[0, 2\pi)$, the matrix $\mathbb{F}_n$ is the so called Fourier matrix of order $n$ given by
 \begin{align}
 (\mathbb{F}_{n})_{i,j}=\frac{1}{\sqrt{n}} \E^{\hat{\imath}(i-1)\theta_{j,n}}, \quad i,j=1,\ldots,n.
 \end{align}

\begin{Definition}\label{def:Cir}
Let the Fourier coefficients of  a given function $f\in L^1(-\pi,\pi)$ be 
\begin{align}\label{eq:coeff_fourier}
  \hat a_j(f):=\frac1{2\pi}\int_{Q}f(\theta){\rm e}^{- \hat{\imath} j \theta} d\theta\in\mathbb{C},
  \qquad  \hat{\imath}^2=-1, \, j\in\mathbb Z.
\end{align}
Then, the ${n}$th
circulant matrix $\mathcal{C}_{n}(f)$ associated with $f$ is given by
 \begin{equation}\label{eq:Circulant_diagonalization}
 \mathcal{C}_{n}(f)=\sum_{|j|<n}\hat{a}_{j}(f)Z_{n}^{j}=\mathbb{F}_{n}  D_{n}(f) \mathbb{F}_{n}^{H},
  \end{equation} where $Z_{n}^{(j)}$ is the $n \times n$ matrix whose $(i,k)$ entry equals 1 if $\mathrm{mod}(i-k,n)=1$ and zero otherwise. 
Moreover,
\begin{equation*}
  D_{n}(f)=\diag\left(s_n(f(\theta_{j,n}))\right),\quad j=1,\ldots,n,
\end{equation*}
where 
$s_{n}(f(\theta))$ is the $ n$th Fourier sum of $f$ given by
\begin{equation*}
s_{n}(f({\theta}))= \sum_{k=1-n}^{n-1}  \hat{a}_{k}(f)
\E^{\iota k\theta}.
\end{equation*}

\end{Definition}

The set $\{\mathcal{C}_n(f)\}_{n\in\mathbb N}$ is called the \textit{family of circulant matrices generated by $f$}, that
in turn is referred to as the \textit{generating function or the symbol of $\{\mathcal{C}_n(f)\}_{n\in\mathbb N}$}. 
Note that $\rho(\mathcal{C}_n(f)) \leq \|f\|_\infty$ for all $n\in\mathbb N$.
The set of circulant matrices of the same size $n$ defines a \emph{matrix algebra} since {it} is closed by sum, product, and inversion.
From a computational point of view, allocating only the vector of the Fourier coefficients $\hat{a}_{j}(f)$, all the computations (matrix-vector product, inversion, etc.) involving the matrix $\mathcal{C}_{n}(f)$ can be computed {using the} FFT. 

\begin{Remark}\label{rmk:eig_circ_pol}
If ${f}$ is a trigonometric polynomial of fixed degree less than $n$,   the entries of $D_n(f)$ are the eigenvalues of $\mathcal{C}_n(f)$, explicitly given by sampling $f$ {using} the grid $\theta_{j,n}$:
 \begin{align*}
  \lambda_j(\mathcal{C}_n(f))&=f\left(\theta_{j,n}\right),\quad j=1,\ldots,n,\nonumber\\
   D_{n}(f)&=\diag\left(f\left(\theta_{j,n}\right)\right),\quad j=1,\ldots,n.
 \end{align*}
\end{Remark}

Given a circulant matrix $A \in \C^{n \times n}$, we denote by $f_A$ its symbol such that $A=\mathcal{C}_n(f_A)$. Moreover, if $f_A$ is a trigonometric polynomial, its degree is denoted by $z_A$ and $A$ is a band matrix with bandwidth $z_A+1$.

\section{Multigrid methods}\label{sec:multigrid}
This section collects relevant results concerning the convergence theory of  algebraic multigrid  methods. 
We first recall the approximation property and then its equivalent condition in terms of the symbol of circulant matrices. 
Next we report the proposal in \cite{MR3439215} for the Stokes problem that will be exploited in the case of circulant matrices in Section~\ref{sec:TGM_conv_circ}.
 
In the general case, we are interested in solving a linear system $X_{n}x_{n}=b_{n}$  where $X_{n}$ is HPD. Assume $k<n$ and define a full-rank rectangular matrix $P_{n,k}\in \mathbb{C}^{n\times k}$, which is used as a grid transfer operator to reduce the problem size.   

\subsection{TGM and approximation property} 
A TGM combines smoothing iterations with a coarse grid correction, which requires the solution of the error equation on a subspace of reduced dimension.  In this paper, we consider only a post smoother, which consists in a single step of the damped Jacobi method with iteration matrix $I_n-\omega D_{X_n}^{-1}X_n.$

The global iteration matrix of TGM is given by 
\begin{small}
\begin{equation*}
{\rm TGM}(X_{n},P_{n,k},\omega)=
(I_n-\omega D_{X_n}^{-1}X_n)
\left[I_{n}-P_{n,k}\left(P_{n,k}^{H}
X_{n}P_{n,k}\right)^{-1}P_{n,k}^{H}X_{n}\right].
\end{equation*}
\end{small}
The convergence results focus on the choice of $P_{n,k}$ and $\omega$ such that the spectral radius of ${\rm TGM}(X_{n},P_{n,k},\omega)$ is strictly less than 1.

Following the Ruge and St\"{u}ben approach \cite{RStub}, later generalized by Notay in \cite{MR3395388, MR3439215, MR3716588}, we introduce the so-called approximation property.

\begin{Definition}\label{def:approxprop}
Let $X_n$ be an HPD $n\times n$ matrix.  Let $P_{n,k}$ a full-rank $n\times k$ matrix, $k<n$. 
If there exists a constant ${\kappa}({X_n,P_{n,k}})\in \mathbb{R}$ such that
\begin{equation}\label{eq:approximation}
\min_{\mathbf{v}\in \mathbb{C}^{k}} \|\mathbf{u}-P_{n,k}\mathbf{v}\|^2_{D_{X_n}}\le {\kappa}({X_n},P_{n,k})\|\mathbf{u}\|^2_{X_n}, \quad \forall \, \mathbf{u}\in \mathbb{C}^{n},
\end{equation}
then the pair $(X_n,P_{n,k})$ is said to fulfil the \emph{approximation property} and the constant ${\kappa}({X_n},P_{n,k})$ is an associated approximation property constant.
\end{Definition}

Starting from the study in \cite{FS1}, many results have been given on the choice of the prolongation and restriction operators for Toeplitz and circulant systems \cite{chan1998,huckle2002}. 

Let $n$ be even, the common approach for a circulant matrix $\mathcal{C}_n(f)$, where $f$ is a nonegative trigonometric polynomial, consists in choosing the grid transfer operator 
\begin{equation*}
P_{\mathcal{C}_n(f)}=\mathcal{C}_{n}(p)K_{n}^T \in \C^{n \times \frac{n}{2}},
\end{equation*}
where the trigonometric polynomial $p$ is chosen according to Lemma \ref{lem:cristina} and the matrix
\begin{equation}\label{eq:def_cutting_matrix_Kn_circ}
K_{n}=\left[\begin{array}{cccccccc}
		1 & 0 & & & & &\\
		  &   & 1 & 0 & & & & \\
			&   &   &   & \ddots & \ddots & & \\
			&   &   &   &        &        & 1 & 0		
		\end{array}\right]_{\frac{n}{2}\times n}.
\end{equation}
is the down-sampling operator.

Using the Galerkin approach, the classical TGM convergence theorem for circulant matrices was proved in \cite{Serra_Possio} for the approximation property as formulated in \cite{RStub}. Here we prove that the same conditions satisfy the approximation property according to  Definition \ref{def:approxprop}.

\begin{Lemma}\label{lem:cristina}
Let $A=\mathcal{C}_n(f)$, with $f$ being a nonnegative trigonometric polynomial such that $f(\theta_0)=0$ and $f(\theta)>0$ for all $\theta \in [0,2\pi)$. Let $P_{A}= \mathcal{C}_n(p)K_n^T$, with $p$ satisfying:
\begin{enumerate}
\item $|p|^2(\theta)+|p|^2(\theta+\pi)>0 \quad \forall \theta\in[0,2\pi),$
\item $\limsup_{\theta\to \theta_0} \frac{|p|^2(\theta+\pi)}{f(\theta)}<c$.
\end{enumerate}
Then the pair $(A,P_A)$ fulfills the approximation property in equation (\ref{eq:approximation}) with 
\begin{equation}\label{eq:consttgm}
{\kappa}({A},P_{A}) = 2 \hat a_0(f)\left\|\frac{|p|^2(\theta+\pi)}{f(\theta)}\right\|_\infty  \left\|\frac{1}{|p|^2(\theta)+|p|^2(\theta+\pi)}\right\|_\infty.
\end{equation}
\end{Lemma}
\begin{proof}
Fixing $\v{v}=(P_A^HP_A)^{-1}P_A^H\v{u}$, the condition \eqref{eq:approximation} is implied by
\[\|\mathbf{u}-P_{n,k}\mathbf{v}\|^2_{D_A}\le {\kappa}(A,P_A)\|\mathbf{u}\|^2_A, \quad \forall \, \mathbf{u}\in \mathbb{C}^{n},\]
which is equivalent to the matrix inequality 
\[\hat a_0(f)(I-P_A(P_A^HP_A)^{-1}P_A^H ) \le {\kappa}(A,P_A) A\]
since $D_{A_n} =\hat a_0(f) I_n$.
By performing a block diagonalization of all the involved matrices (see \cite{Serra_Possio}), to have  \eqref{eq:approximation}, it is enough to prove
\begin{equation*}
{\kappa}({A},P_{A})\begin{bmatrix}
&f(\theta)& 0\\
& 0& f(\theta+\pi)\end{bmatrix} \ge \frac{\hat a_0(f)}{|p|^2(\theta)+|p|^2(\theta+\pi)}\begin{bmatrix}
&|p|^2(\theta+\pi) & -p(\theta)p(\theta+\pi)\\
& -p(\theta)p(\theta+\pi) & |p|^2(\theta)
\end{bmatrix}, 
\end{equation*}
for all grid point $\theta \in \left\{\theta_{j,n}=\frac{(j-1)2\pi}{n} \, \big| \; j=1,\dots,n\right\}$.
In conclusion, ${\kappa}({A},P_{A})$, defined as in \eqref{eq:consttgm}, satisfies 
\[{\kappa}({A},P_{A}) \ge  \frac{\hat a_0(f)}{|p(\theta)|^2+|p(\theta+\pi)|^2} 
\left( \frac{|p|^2(\theta+\pi)}{f(\theta)} + \frac{|p|^2(\theta)}{f(\theta+\pi)}\right), \qquad \forall \theta\in[0,2\pi),\]
and hence the inequality \eqref{eq:approximation} is true.
\end{proof}

\begin{Remark}\label{rem:singular}
Let $A=\mathcal{C}_n(f)$, if $f$ vanishes at a grid point $\theta_0=\theta_{j,n}=\frac{(j-1)2\pi}{n}, j=1,\dots,n$, then $A$ is singular and the approximation property in Definition \ref{def:approxprop} cannot be applied. Nevertheless, Lemma \ref{lem:cristina} still holds applying a small rank one correction to $A$ obtaining a HPD matrix as in \cite{ADS}, or the system matrix is ``naturally'' singular in the sense of \cite[section 3.2]{MR3553931} and the kernel is in the range of the prolongation.  
\end{Remark}

\subsection{TGM for the Stokes problem}\label{ssec:tgm_stokes}
Considering the Stokes problem, the TGM proposed in \cite{MR3439215} converges thanks to the following result.

\begin{Theorem}[\cite{MR3439215}]\label{thm:teo_notay}
Let 
\begin{equation}\label{eq:saddle_notay}
\mathcal{A}=	\begin{bmatrix}
		A & B^T \\
		B & -C
	\end{bmatrix}
\end{equation}
be a matrix such that $A$ is an $n\times n$ HPD matrix and $C$ is an $m\times m$ nonnegative definite matrix. Assume that $B$ has rank $m\leq n$ or that $C$ is positive definite on the null space of $B^T$. 

Let 
$\alpha$ be a positive number such that $\alpha<2(\lambda_{\max}(D_A^{-1}A))^{-1}$ and define 
\begin{equation}\label{eq:LAU}
\hat{\mathcal{A}}=\mathcal{LAU}, \qquad
\mathcal{L}= \begin{bmatrix}
I_{n}\\
\alpha B D_{A}^{-1}&-I_{m}
\end{bmatrix}, \qquad \mathcal{U}=\begin{bmatrix}
I_{n}&-\alpha D_{A}^{-1}B^T\\
&I_{m}
\end{bmatrix},
\end{equation}
and
 \begin{equation}\label{eq:Chat}
  \hat{C}=C+B(2\alpha D_A^{-1}-\alpha^2D_A^{-1}AD_A^{-1})B^T.
 \end{equation}  
Let $P_A$ and $P_{\hat{C}}$ be, respectively, $n\times k$ and $m\times \ell$ matrices of rank $k<n$ and $\ell<m$, define the prolongation 
\begin{equation}\label{eq:PNotay}
\mathcal{P}=\begin{bmatrix}
P_{A}& \\
& P_{\hat{C}}
\end{bmatrix}
\end{equation}
for the global system involving $\mathcal{\hat{A}}$ and suppose that  the pairs $(A,P_A)$ and $(\hat{C}, P_{\hat{C}})$ fulfill the approximation property~(\ref{eq:approximation}).
    
  Then, the spectral radius of the TGM iteration matrix {using} one iteration of the damped Jacobi method with relaxation parameter $\omega$ {as  post smoother} satisfies
\begin{equation}\label{eq:notay_max}
\rho({\rm TGM}(\mathcal{\hat{A}},\mathcal{P},\omega))\le \max\left(1-\frac{\omega}{{\kappa}({A},P_{A})},  \, 1-\frac{\omega}{{\kappa}({\hat{C}},P_{\hat{C}})}, \, \omega \hat{\gamma}_A-1,  \, \omega \hat{\gamma}_{\hat{C}}-1,  \, \sqrt{1-\frac{\omega(2-\omega\tilde{\gamma})}{\tilde{\kappa}}} \right),
\end{equation}
where \begin{align*}
\hat{\gamma}_{A}= \left(\alpha\left(2-\alpha\lambda_{\max}(D_{A}^{-1}A)\right)\right)^{-1},& \qquad \hat{\gamma}_{\hat{C}}= \lambda_{\max}\left(D_{\hat{C}}^{-1}(C+BA^{-1}B^T)\right),  \\
\tilde{\gamma}=\frac{2\hat{\gamma}_{A}\hat{\gamma}_{\hat{C}}}{\hat{\gamma}_{A}+\hat{\gamma}_{\hat{C}}},& \qquad \tilde{\kappa}=\frac{2{\kappa}({A},P_{A}){\kappa}({\hat{C}},P_{\hat{C}})}{{\kappa}({A},P_{A})+{\kappa}({\hat{C}},P_{\hat{C}})}.
\end{align*}

\end{Theorem}

The transformed linear system with coefficient matrix $\hat{\mathcal{A}}$ defined in \eqref{eq:LAU} allows to study separately the approximation property for the two matrices $A$ and $  \hat{C}$, which is HPD thanks to the choice of $\alpha$. All 
$\alpha\in\left(0,2/\lambda_{\max}(D_A^{-1}A)\right)$ ensure the convergence of TGM, but $ \hat{\gamma}_{A}$, and hence $\rho({\rm TGM}(\mathcal{\hat{A}},\mathcal{P},\omega))$, depends on $\alpha$. A good compromise is to choose  $\alpha \approx \lambda_{\max}(D_A^{-1}A)^{-1}$ which minimizes  $\hat{\gamma}_{A}$.
 On the other hand, {for} fixed $\alpha$, the relaxation parameter $\omega$ of the post smoother should be chosen in order to minimize the bound on
 $\rho({\rm TGM}(\mathcal{\hat{A}},\mathcal{P},\omega))$, i.e., the maximum in inequality \eqref{eq:notay_max}.
 
\subsection{Multigrid methods}

The TGM is useful for practical and preliminary convergence analysis of a multigrid method, but in practical applications, even the coarser problem is too large to be solved directly. Hence, a simple strategy is to apply recursively the same algorithm at the coarser error equation instead of solving it directly. Such {a} recursive application, until a small size problem is obtained, is known as V-cycle. A more robust multigrid method can be obtained concatenating two recursive calls resulting in the so-called W-cycle. 

Since the solutions of the error equations at the coarser levels are only approximated, the TGM convergence study is necessary but not sufficient to have a robust multigrid method. A more robust {result}, known as \emph{level independence}, is {obtained by} applying a TGM at a generic level of the multigrid hierarchy. This allows to obtain a linear convergence rate for the W-cycle but it is still not enough for the V-cycle {\cite{ADS,RStub}}.
Therefore further algebraic tools have been introduced for proving the V-cycle optimality, see e.g. {\cite{napov2011,RStub}}.

For circulant matrices, using the Galerkin approach, the following lemma states that the circulant structure is preserved at the coarser levels with a symbol depending explicitly on the restriction $K_{n}\mathcal{C}_{n}(p_1)^H$ and the prolongation $\mathcal{C}_{n}(p_2)K_{n}^T$.
 
\begin{Lemma}\cite[Proposition 6]{D10}\label{lem:f_coarse} 
Let $f$ be a trigonometric polynomial. Let $K_{n}$ be defined as in formula (\ref{eq:def_cutting_matrix_Kn_circ}). Then the matrix
$\left(\mathcal{C}_{n}(p_1)K_{n}^T\right)^H\mathcal{C}_{n}(f)\left(\mathcal{C}_{n}(p_2)K_{n}^T\right)\in\mathbb{C}^{k\times k}$, $k=\frac{n}{2}$,
coincides with $\mathcal{C}_{k}(\hat{f})$ where
\begin{align}\label{eq:f_coarse}
  \hat{f}(\theta)=\frac{1}{2}\left(\overline{p}_1 f p_2\left(\frac{\theta}{2}\right)+
	\overline{p}_1 f p_2\left(\frac{\theta}{2}+\pi\right)\right).
\end{align} 
Moreover, if $\theta_0 \in [0,2\pi)$ is a zero of $f$ and the projectors $p_1$ and $p_2$ satisfy the conditions 1 and 2 in Lemma \ref{lem:cristina}, then $\hat{\theta}_0=2\theta_0 \, {\rm mod} \, 2\pi$ is a zero of $\hat{f}$ and the two zeros have the same order.
\end{Lemma}

The properties of the coarser symbols are crucial to study the multigrid convergence. Indeed, thanks to Lemma \ref{lem:f_coarse}, it is possible to prove the level independence under the same TGM assumptions in Lemma~\ref{lem:cristina}. While for the V-cycle optimality, the condition~2 on the projector  has to be replaced with the stronger condition $\limsup_{\theta\to \theta_0} \frac{p(\theta+\pi)}{f(\theta)}<c$ (see~\cite{ADS}). 

For saddle point problems, at least to our knowledge, there are no results in the literature that extend the general algebraic TGM analysis in Theorem~\ref{thm:teo_notay} to multigrid methods. Therefore, in Section~\ref{sec:procedure}, we will provide a multigrid analysis in the case of circulant blocks. 

\section{TGM convergence for saddle point matrices with circulant blocks}\label{sec:TGM_conv_circ}

In this section, we prove how to take advantage of the result in Theorem \ref{thm:teo_notay} in the case where the blocks in the saddle-point problem~\eqref{eq:linear_system}-\eqref{eq:matrix_saddle} involve circulant matrices. In this case, the following preliminary lemma is useful to obtain an HPD matrix $\hat{C}$ in \eqref{eq:Chat}.

\begin{Lemma}\label{lem:alpha}
Let $A=\mathcal{C}_n(f)$ with $f\ge0$. If $0<\alpha< \frac{2 \hat{a}_0(f)}{\|f\|_\infty}$, then
the matrix $2\alpha D^{-1}_{A}-\alpha^2D^{-1}_A A D^{-1}_A $ is HPD.  
\end{Lemma}
\begin{proof}
Since $A=\mathcal{C}_n(f)$ with $f\ge0$, then $D^{-1}_A=\frac{1}{ \hat{a}_0(f)} I_n$ with $ \hat{a}_0(f)>0$. 
Moreover, the matrix $2\alpha D^{-1}_{A}-\alpha^2D^{-1}_A A D^{-1}_A$ is circulant and generated by the real function 
$g(\theta)= \frac{2\alpha}{\hat{a}_0(f)}-\frac{\alpha^2}{\hat{a}^2_0(f)} f(\theta)$, which is positive for all $\theta$
if $\alpha \in \big(0,\frac{2 \hat{a}_0(f)}{f(\theta)}\big)$.
\end{proof}


The following theorem provides a deeper analysis of Theorem \ref{thm:teo_notay} using the symbol analysis of circulant matrices. In particular, it allows to define optimal projectors and smoothers as numerically confirmed in the example in Subsection~\ref{sec:numerical_TGM}. This is the first step towards a multigrid analysis provided in Section~\ref{sec:procedure}.

\begin{Theorem}\label{thm:Notay_symbol}
Let us define the matrix 
\begin{equation}\label{eq:saddle_notaycirco}
\mathcal{A}=	\begin{bmatrix}
		A & B^T \\
		B & -C
	\end{bmatrix}_{2n\times 2n},
\end{equation}
where
\begin{itemize}
\item  $A=\mathcal{C}_n(f_A)$ with $f_A$ trigonometric polynomial such that $f_A(\theta_0)=0$ and $f_A(\theta)>0$ for all $\theta\in [0,2\pi]\setminus \{\theta_0\} $;
\item  $C=\mathcal{C}_n(f_C)$ with $f_C$ nonnegative trigonometric polynomial;
\item $B=\mathcal{C}_n(f_B)$ with  $f_{B}$  trigonometric polynomial such that  $|f_{B}|^2({\theta}_0)=0$ and $|f_{B}|^2(\theta)>0$ for all 
$\theta\in [0,2\pi]\setminus \{\theta_0\} $,  and $\limsup_{\theta \to \theta_0} \frac{|f_{B}|^2(\theta)}{f_A(\theta)}< \infty$.
\end{itemize}
Let $\alpha$ be a positive number such that $\alpha< \frac{2 \hat{a}_0(f_A)}{\|f_A\|_\infty}$ and define $\mathcal{\hat{A}}=\mathcal{LAU}$ such as in Theorem \ref{thm:teo_notay}.
Let $\hat{C}$ be defined as in \eqref{eq:Chat}, then $\hat{C}=\mathcal{C}_n(f_{\hat{C}})$ with 
\begin{equation}\label{eq:f_Chat} 
f_{\hat{C}}(\theta)=f_{C}(\theta) + \frac{\alpha |f_B|^2(\theta)}{\hat{a}_0(f_A)} \left(2- \frac{\alpha}{\hat{a}_0(f_A)} f_A(\theta) \right).
\end{equation}
Let $\mathcal{P}$ be defined as in \eqref{eq:PNotay}, where
$P_A=\mathcal{C}_n(p_A)K_n^T$ and $ P_{\hat{C}}=\mathcal{C}_n(p_{\hat{C}})K_n^T$, with $p_A$ and $p_{\hat{C}}$ trigonometric polynomials.
  
   Consider a TGM associated with one iteration of damped Jacobi as postsmoothing with relaxation parameter $\omega$. If
\begin{enumerate}
\item $|p_A|^2(\theta)+ |p_A|^2(\theta+\pi)>0$ and $|p_{\hat{C}}|^2(\theta)+ |p_{\hat{C}}|^2(\theta+\pi)>0$  for all $\,\theta\in [0,2\pi];$ 
\item $\lim\sup_{\theta\to \theta_0} \frac{|p_A|^2(\theta+\pi)}{f_A(\theta)}<\infty$ and $\lim\sup_{\theta\to {\theta_0}} \frac{|p_{\hat{C}}|^2(\theta+\pi)}{f_{\hat{C}}(\theta)}<\infty;$
\item $\omega<2 \min \left( 2\alpha -\frac{\alpha^2}{\hat{a}_0(f_A)}\|f_A\|_\infty\, ,  \, \hat{a}_0(f_{\hat{C}})\left\|f_C+\frac{|f_B|^2}{f_A}\right\|_\infty^{-1}\right) $.

\end{enumerate}
Then, 
\[\rho({\rm TGM}(\mathcal{\hat{A}},\mathcal{P},\omega))<1.\]
\end{Theorem}
\begin{proof}
For the sake of simplicity, let us assume that $\theta_0$ is different from each grid point $ \theta_{i,n}$, $i=1,\dots,n$ (for the case of $\theta_0$ equal to a grid point see Remark \ref{rmk:right-hand}), then $A$ is HPD. Analogously, $C$ is nonnegative definite and $B$ is a full-rank matrix. Since $\alpha< \frac{2 \hat{a}_0(f_A)}{\|f_A\|_\infty}$, thanks to Lemma \ref{lem:alpha} and the nonnegativity of $C$, we have that $\hat{C}$ is HPD. Hence,  thanks to hypotheses 1., 2., and Lemma \ref{lem:cristina}, the pairs $(A, P_A)$ and $({\hat{C}},P_{\hat{C}})$ fulfil the approximation property (\ref{eq:approximation}). 
Therefore, we can apply Theorem \ref{thm:teo_notay}, which implies that the TGM applied to $\mathcal{\hat{A}}$ satisfies the inequality 
\eqref{eq:notay_max}, i.e.,
$$\rho({\rm TGM}(\mathcal{\hat{A}},\mathcal{P},\omega))\le \max\left(1-\frac{\omega}{{\kappa}({A},P_{A})}, 1-\frac{\omega}{{\kappa}({\hat{C}},P_{\hat{C}})}, \omega \hat{\gamma}_A-1, \omega \hat{\gamma}_{\hat{C}}-1, \sqrt{1-\frac{\omega(2-\omega\tilde{\gamma})}{\tilde{\kappa}}} \right).$$
In conclusion, in order to prove that $\rho({\rm TGM}(\mathcal{\hat{A}},\mathcal{P},\omega))<1$, we need to prove that each  of the quantities
in the maximum is bounded from above by a constant strictly smaller that 1.
 
 The quantities $1-\frac{\omega}{{\kappa}({A},P_{A})}$ and $1-\frac{\omega}{{\kappa}({\hat{C}},P_{\hat{C}})}$ are strictly smaller than 1 because the approximation property constants ${{\kappa}({A},P_{A})}$ and ${{\kappa}({\hat{C}},P_{\hat{C}})}$  are finite and positive.
 
Concerning the two terms $\omega \hat{\gamma}_A-1$ and  $\omega \hat{\gamma}_{\hat{C}}-1$ in the maximum, we estimate
 \begin{equation}\label{eq:gamma1}
 \hat{\gamma}_A= \left(\alpha\left(2-\alpha\lambda_{\max}(D_{A}^{-1}A)\right)\right)^{-1}
 \le \left(2\alpha -\frac{\alpha^2}{\hat{a}_0(f_A)}\|f_A\|_\infty\right)^{-1} ,
 \end{equation}   
  \begin{equation}\label{eq:gamma2}
\hat{\gamma}_{\hat{C}}= \lambda_{\max}\left(D_{\hat{C}}^{-1}(C+BA^{-1}B^T)\right)
\le \frac{\left\|f_C+\frac{|f_B|^2}{f_A}\right\|_\infty}{\hat{a}_0(f_{\hat{C}})}, 
\end{equation}
where in the latter inequality we are using $\limsup_{\theta \to \theta_0} \frac{|f_{B}|^2(\theta)}{f_A(\theta)}< \infty$ and hence the generating function $f_C+\frac{|f_B|^2}{f_A}$ belongs to $L^1([-\pi,\pi])$.
Hence, using the majorization of $\omega$ in the hypothesis 1., it holds
\begin{equation*}
\omega \hat{\gamma}_A-1< 2 \left(2\alpha -\frac{\alpha^2}{\hat{a}_0(f_A)}\|f_A\|_\infty\right)  \left(2\alpha -\frac{\alpha^2}{\hat{a}_0(f_A)}\|f_A\|_\infty\right)^{-1}-1=1,
\end{equation*}
\begin{equation*}
\omega \hat{\gamma}_{\hat{C}}-1<2 \hat{a}_0(f_{\hat{C}})\left\|f_C+\frac{|f_B|^2}{f_A}\right\|_\infty^{-1} \frac{1}{\hat{a}_0(f_{\hat{C}})}\left\|f_C+\frac{|f_B|^2}{f_A}\right\|_\infty -1=1. 
\end{equation*}

Finally, in order to prove that $\sqrt{1-\frac{\omega(2-\omega\tilde{\gamma})}{\tilde{\kappa}}} <1$, we prove that $2-\omega \tilde{\gamma}>0$.
From hypothesis 3., we have
\begin{equation}\label{eq:om1}
\omega
< \left(2\alpha -\frac{\alpha^2}{\hat{a}_0(f_A)}\|f_A\|_\infty\right)+\hat{a}_0(f_{\hat{C}})\left\|f_C+\frac{|f_B|^2}{f_A}\right\|_\infty^{-1}.
\end{equation}
Moreover, using the estimations \eqref{eq:gamma1} and \eqref{eq:gamma2}, we have
\begin{equation}\label{eq:gamma3}
\tilde{\gamma}=\frac{2\hat{\gamma}_{A}\hat{\gamma}_{\hat{C}}}{\hat{\gamma}_{A}+\hat{\gamma}_{\hat{C}}}\le 2 \left(
\frac{\left(2\alpha -\frac{\alpha^2}{\hat{a}_0(f_A)}\|f_A\|_\infty\right)^{-1} \frac{1}{\hat{a}_0(f_{\hat{C}})}\left\|f_C+\frac{|f_B|^2}{f_A}\right\|_\infty}{\left(2\alpha -\frac{\alpha^2}{\hat{a}_0(f_A)}\|f_A\|_\infty\right)^{-1}+\frac{1}{\hat{a}_0(f_{\hat{C}})} \left\|f_C+\frac{|f_B|^2}{f_A}\right\|_\infty}
\right),
\end{equation}
where we used the fact that the function $(x,y)\mapsto 2 \frac{x y}{x+y} $ is increasing in $(0,+\infty)\times (0,+\infty)$.
Then, combining \eqref{eq:om1} and \eqref{eq:gamma3}, we obtain $\omega\tilde{\gamma}< 2$.
\end{proof}

\begin{Remark}\label{rmk:right-hand}
Similar to Lemma \ref{lem:cristina}, the proof of the previous theorem requires that the symbols $f_A$ and $f_B$ do not vanish at a grid point. Otherwise, the result still holds applying one of the two techniques in Remark \ref{rem:singular}.
\end{Remark}


\subsection{Example: 1D elasticity problem}\label{sec:numerical_TGM}
In this subsection we want to show the numerical efficiency of the TGM convergence results in Theorem \ref{thm:Notay_symbol} when applied to linear systems stemming from the finite difference approximation of a one dimensional elasticity problem. 

Consider the coupled system of one-dimensional scalar equations
\begin{equation}\label{eq:1D_elasticity}
\begin{cases}
-u''-v' &= g_1(x),\\
u'-\rho v &=  g_2(x),
\end{cases}
\end{equation}
with $x \in \Omega=[0,1]$, $\rho>0$, and periodic boundary conditions. 
Discretizing the problem using standard finite difference methods with stepsize $h=1/(n+1)$ and scaling by the diagonal matrix $\mathcal{D}^{(1)}$, we obtain
\begin{equation} \label{eq:spectrally_analyzed_matrix}
\mathcal{A} = \mathcal{D}^{(1)}	\begin{bmatrix}
		A & hB^T \\
		hB & -h^2 C
	\end{bmatrix}\mathcal{D}^{(1)} = 	\begin{bmatrix}
		A& B^T \\
		B & -C
	\end{bmatrix},
	\qquad
	\mathcal{D}^{(1)}  = \begin{bmatrix}
		I & O  \\
		O & \frac{1}{h}I \\
	\end{bmatrix},	
\end{equation}
where $A,$ $C$, and $B$ are circulant matrices defined by the symbols
\begin{align}\label{eq:ABC_circ}
	f_A(\theta) = 2-2\cos(\theta), \quad f_B(\theta) = 1-{\rm e}^{\hat{\imath}\theta}, 
	\quad f_C(\theta) = \frac{2\rho}{3} (2+\cos(\theta)).
\end{align}

We prove that these symbols fulfil the hypothesis of Theorem  \ref{thm:Notay_symbol}. 
In particular, $f_C>0$ and $|f_B|^2= f_A$, which vanishes in $\theta_0=0$ and is positive in $(0, 2\pi)$. 

Since $\hat{a}_0(f_A)=2$ and $\|f_A\|_\infty=4$, we have
$\frac{2 \hat{a}_0(f_A)}{\|f_A\|_\infty}=1$
and hence we choose  $\alpha=\frac{1}{2}$, which is the middle point of the admissible interval $(0,1)$.

According to \eqref{eq:f_Chat}, the symbol of $\hat{C}$ is
\begin{align*}
f_{\hat{C}}(\theta)=\frac{2\rho}{3} (2+\cos(\theta)) + \frac{\alpha(2-2\cos(\theta))}{2} \left(2- \frac{\alpha}{2} (2-2\cos(\theta) \right),\\
\end{align*}
that for $\alpha=\frac{1}{2}$ is
\begin{equation}\label{eq:fCelast}
f_{\hat{C}}(\theta)
=\frac{32\rho+15}{24}+\frac{4\rho-3}{6} \cos\theta-\frac{1}{8}\cos 2\theta.
\end{equation}
Concerning the grid transfer operators, the symbol 
\[p_A(\theta)=\sqrt{2}(1+\cos(\theta)) \] 
fulfills the assumptions 1.\ and 2.\ of Theorem \ref{thm:Notay_symbol} because
\begin{align*}
|p_A|^2(\theta)+ |p_A|^2(\theta+\pi)&=
8 + 8 \cos^2(\theta)>0, \qquad \forall \theta \in [0,2\pi], \\
\limsup_{\theta\to 0} \frac{|p_A|^2(\theta+\pi)}{f_A(\theta)}&=\limsup_{\theta\to 0} \frac{|2-2\cos(\theta)|^2}{2-2\cos(\theta)}=0.
\end{align*}
On the other hand, $p_{\hat{C}}$ could be chosen as the trivial downsampling operator since $f_{\hat{C}} > 0$ for all $\rho$, but in practice 
\begin{equation}\label{eq:limit_1_over_fC}
\limsup_{\theta\to 0} \frac{1}{f_{\hat{C}}(\theta)}=\frac{1}{2\rho}.
\end{equation}
Nevertheless, when $\rho$ approaches zero the previous limit goes to infinity and hence, when $\rho$ is small, it could be useful to choose $p_{\hat{C}}$ such that $p_{\hat{C}}(\pi)=0$.
Therefore, we choose 
\[p_{\hat{C}}(\theta) = p_A(\theta)=\sqrt{2}(1+\cos(\theta)) \] 
such that
$$\limsup_{\theta\to 0} \frac{|p_{\hat{C}}|^2(\theta+\pi)}{f_{\hat{C}}(\theta)}=0, \qquad \forall \rho>0.$$

To fulfill the assumption 3.\ of Theorem \ref{thm:Notay_symbol} we require
 that 
\begin{align*}
\omega&<2 \min\left( 2\alpha -\frac{\alpha^2}{\hat{a}_0(f_A)}\|f_A\|_\infty\, ,  \, \hat{a}_0(f_{\hat{C}})\left\|f_C+\frac{|f_B|^2}{f_A}\right\|_\infty^{-1}\right).
\end{align*}
For $f_{\hat{C}}$ defined as in \eqref{eq:fCelast}, it holds
\[
\hat{a}_0(f_{\hat{C}})\left\|f_C+\frac{|f_B|^2}{f_A}\right\|_\infty^{-1}=\, \frac{32\rho+15}{48\rho+24}\, .
\]
and hence, since $\alpha=\frac{1}{2}$, we have
\begin{align*}
\omega&<2 \min \left(
\frac{1}{2}\,, \, \frac{32\rho+15}{48\rho+24}\right) = 1.
\end{align*}
Therefore, Theorem \ref{thm:Notay_symbol} ensures that the TGM applied to the system having $\mathcal{\hat{A}}$ as coefficient matrix converges,
but it remains to estimate the best $\omega\in
(0,1)
$.
This can be done minimizing the upper bound in (\ref{eq:notay_max}). Note that such  $\omega_{\rm opt}$  could be different from the value that minimizes $\rho({\rm TGM}(\mathcal{\hat{A}},\mathcal{P},\omega))$.
Nevertheless, the numerical results confirm that it gives the minimum number of iterations to convergence.  

We already provided upper bounds for the quantities $\hat{\gamma}_A,$ $\hat{\gamma}_{\hat{C}}$, and $\tilde{\gamma}$ in the proof of Theorem \ref{thm:Notay_symbol}.
Precisely,
 \begin{align*}
 \hat{\gamma}_A\le \left(2\alpha -\frac{\alpha^2}{\hat{a}_0(f_A)}\|f_A\|_\infty\right)^{-1}=
 2
 ,
 \end{align*}   
  \begin{align*}
\hat{\gamma}_{\hat{C}}\le
\frac{1}{\hat{a}_0(f_{\hat{C}})}\left\|f_C+\frac{|f_B|^2}{f_A}\right\|_\infty= 
\frac{48\rho+24}{32\rho+15}=\frac{48}{31}
, 
\end{align*}
 and 
 $\tilde{\gamma}\le 96/55.$
 
In the remaining part of the section we fix $\rho=1/2$ in order to show how to make the explicit computation of the optimal parameter $\omega$
. Firstly, we focus on  $\kappa(A,P_A)$,  $\kappa(\hat{C},P_{\hat{C}})$ and $\tilde{k}$. In particular, we use the fact that the pairs $\left(\mathcal{C}_n(f_A),\mathcal{C}_n(p_A)K_n^T\right)$ and $\left(\mathcal{C}_n(f_{\hat{C}}),\mathcal{C}_n(p_{\hat{C}})K_n^T\right)$ satisfy the approximation property (\ref{eq:approximation}) and, by Lemma \ref{lem:cristina},
\begin{align*}
{\kappa}(\mathcal{C}_n(f_A),\mathcal{C}_n(p_A)K_n^T)&\le 2 \hat a_0(f_A)\left\|\frac{|p_A|^2(\theta+\pi)}{f_A(\theta)}\right\|_\infty  \cdot\left\|\frac{1}{|p_A|^2(\theta)+|p_A|^2(\theta+\pi)}\right\|_\infty\\
&= 4 \cdot \left\| \frac{(2+2\cos(\theta+\pi))^2}{2(2-2\cos(\theta))} \right\|_{\infty}\cdot \left\| \frac{2}{(2+2\cos(\theta))^2+(2-2\cos(\theta))^2} \right\|_{\infty}\\
&= 4\cdot \left\| \frac{(2-2\cos(\theta))}{2} \right\|_{\infty} \cdot \left\| \frac{2}{12+4\cos(2\theta)} \right\|_{\infty}  = 4\cdot 2 \cdot \frac{1}{4}=2.
\end{align*}
Since 
$\hat a_0(f_{\hat{C}})=31/24$
, 
\begin{align*}
{\kappa}(\mathcal{C}_n(f_{\hat{C}}),\mathcal{C}_n(p_{\hat{C}})K_n^T)&\le 2 \hat a_0(f_{\hat{C}})\left\|\frac{|p_{\hat{C}}|^2(\theta+\pi)}{f_{\hat{C}}(\theta)}\right\|_\infty  \cdot\left\|\frac{1}{|p_{\hat{C}}|^2(\theta)+|p_{\hat{C}}|^2(\theta+\pi)}\right\|_\infty\\
& = 
\frac{31}{24}
\cdot \left\| \frac{
(2-2\cos\theta)^2
}{
\frac{31}{24}-\frac{1}{6}\cos(\theta)-\frac{1}{8}\cos(2\theta)
} \right\|_{\infty}\cdot 
\frac{1}{4}
=
\frac{31}{8}
\end{align*}
and 
$\tilde{k}\le 124/47$
.
 
Then we need to minimize the function 
\begin{equation}\label{eq:omega}
\mu:\omega \mapsto \max \left( 1- \frac{\omega}{2}, 1- 
\frac{31}{8}
\omega, 
2
\omega-1, 
\frac{48}{31}
\omega-1, 
\sqrt{\frac{1128}{1705}\omega^2 - \frac{47}{62}\omega+1}
\right)
\end{equation} 
in order to choose the relaxation parameter $\omega$ such that the TGM convergence is as fastest as possible. Figure \ref{fig:plot_omega} depicts the five functions in (\ref{eq:omega}) and the minimizer of $\mu$, denoted as $\omega_{\rm opt}$ and computed as the minimum of the parabola 
$\frac{1128}{1705}\omega^2 - \frac{47}{62}\omega+1$. 
This parameter $\omega_{\rm opt}=\frac{55}{96}$ gives the upper bound 
$$\rho({\rm TGM}(\mathcal{\hat{A}},\mathcal{P},\omega_{\rm opt}))\le
\sqrt{\frac{1128}{1705}\omega_{\rm opt}^2 - \frac{47}{62}\omega_{\rm opt}+1}\approx 0.8848$$
and hence the TGM has a linear convergence since the bound of the spectral radius does not depend on the matrix size.

\begin{figure}[htb]
\caption{Plot of the functions in (\ref{eq:omega}). }
\centering\includegraphics[width=\textwidth]{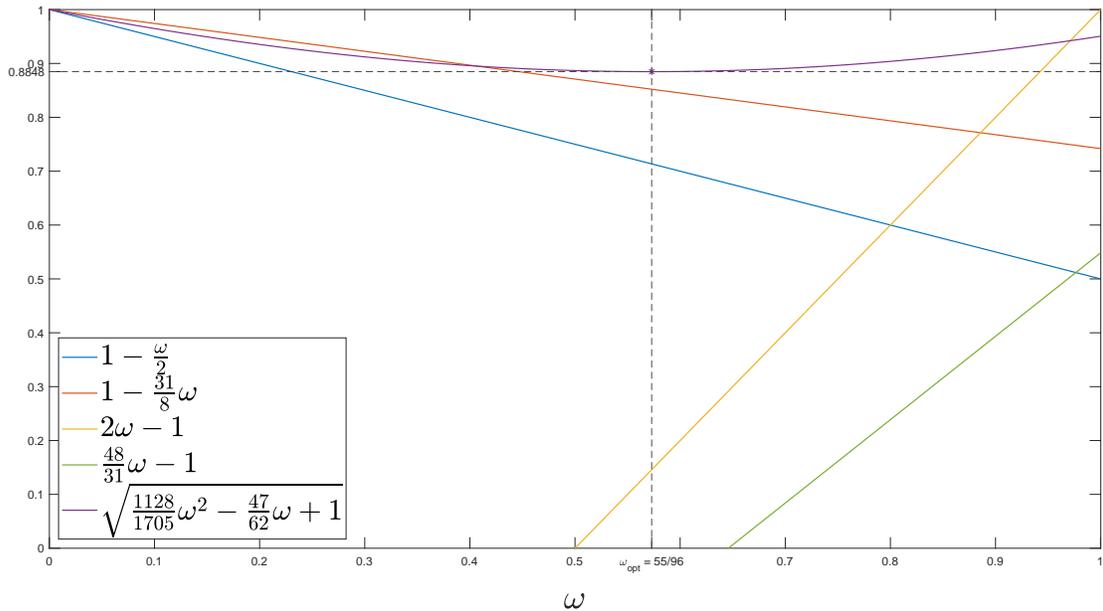}
\label{fig:plot_omega}
\end{figure}

%

\section{Multigrid analysis}\label{sec:procedure}

In this section, we extend the previous TGM convergence to the multigrid method. In particular, for each recursion level of the multigrid method, firstly we apply a diagonal scaling to preserve the same structure of the coefficient matrices, then we prove the TGM convergence (level independence) and the band structure of the involved matrices  keeping a linear cost of the matrix-vector product. It follows that the W-cycle has a constant convergence rate thanks to an automatic estimation of the smoothing parameters.

For the TGM defined in Theorem \ref{thm:teo_notay}, after the projection by $\mathcal{P}$ the coarser matrix has the same structure of $\mathcal{A}$  in (\ref{eq:saddle_notay}) except for the sign of the last block row. Therefore, changing this sign by a left diagonal scaling we can apply recursively the TGM.
In detail, fix the finer level matrix as
\begin{equation}\label{eq:Astep0}
\mathcal{A}_{2n_0}\{0\} = \begin{bmatrix}
	{A}\{0\} & B\{0\}^T \\
	B\{0\} & -C\{0\}
	\end{bmatrix}=\\
	\left[\begin{array}{ccc}
	{A} & B^T \\
	B & -C
	\end{array}\right],
\end{equation}
with $A,B,C \in \C^{n_0 \times n_0}$ for $n_0=2^\beta$, $\beta \in \N$, and $\hat{C}\{0\}=\hat{C}$.
For each level $\ell \geq 0$, let
\[
\mathcal{A}_{2n_{\ell}}\{\ell\} = \begin{bmatrix}
	{A}\{\ell\} & B\{\ell\}^T \\
	B\{\ell\} & -{C}\{\ell\}
	\end{bmatrix}
\]
and compute the transformation \eqref{eq:LAU}, i.e.,
\begin{equation}\label{eq:hatAl}
\mathcal{\hat{A}}_{2n_{\ell}}\{\ell\} = \mathcal{L}\{\ell\}\mathcal{A}_{2n_{\ell}}\{\ell\}\mathcal{U}\{\ell\},
\end{equation}
where
\begin{equation*}
\mathcal{L}\{\ell\}= \begin{bmatrix}
I_{n_{\ell}}\\
\alpha_{\ell} B\{\ell\} D_{A\{\ell\}}^{-1}&-I_{n_{\ell}}
\end{bmatrix}_{2n_{\ell}\times 2n_{\ell}}, \quad
\mathcal{U}\{\ell\}=\begin{bmatrix}
I_{n_{\ell}}&-\alpha_\ell D^{-1}_{{A\{\ell\}}}B\{\ell\}^T\\
&I_{n_{\ell}}
\end{bmatrix}_{2n_{\ell}\times 2n_{\ell}}.
\end{equation*}

 Then we define recursively the sequence of matrices 
\begin{align}\label{eq:recursive_AL}
\mathcal{A}_{2n_{\ell+1}}\{\ell+1\} 
	&= 
	\begin{bmatrix}
		I_{n}& \\
 		& -I_{n}
	\end{bmatrix}
	\begin{bmatrix}
		{P}_{{A}\{\ell\}}& \\
 		& -{P}_{\hat{C}\{\ell\}}
	\end{bmatrix}^T
	\mathcal{\hat{A}}_{2n_{\ell}}\{\ell\} 	
	\begin{bmatrix}
		{P}_{{A}\{\ell\}}& \\
 		& -{P}_{\hat{C}\{\ell\}}
	\end{bmatrix},
\end{align}
where
\begin{small}
\begin{equation}\label{eq:recursive_AL_elements}
\begin{split}
{A}\{\ell+1\}&={P}_{{A}\{\ell\}}^T{A}\{\ell\}{P}_{{A}\{\ell\}},\\
B\{\ell+1\}&={P}_{\hat{C}\{\ell\}}^TB\{\ell\}(I_{n_\ell}-\alpha_\ell D_{{A}\{\ell\}}^{-1}{A}\{\ell\}){P}_{{A}\{\ell\}}, \\
{C}\{\ell+1\}&={P}_{\hat{C}\{\ell\}}^T\hat{C}\{\ell\}{P}_{\hat{C}\{\ell\}},\\
\hat{C}\{\ell\}&={C}\{\ell\}+B\{\ell\} (2\alpha_{\ell} D_{{A}\{\ell\}}^{-1}-\alpha_{\ell}^2D_{{A}\{\ell\}}^{-1}{A}\{\ell\}D_{{A}\{\ell\}}^{-1})B\{\ell\}^T.
\end{split}
\end{equation}
\end{small}
In formula (\ref{eq:recursive_AL_elements}) the matrices
 ${P}_{{A}\{\ell\}}\in\mathbb{R}^{n_{\ell} \times n_{\ell+1}}$ and ${P}_{\hat{C}\{\ell\}}\in\mathbb{R}^{n_{\ell}\times n_{\ell+1}}$,
 with $n_{\ell+1}=n_{\ell}/2$, are the prolongation operators chosen for solving efficiently the scalar systems with coefficient matrix   ${A}\{\ell\}$ and $\hat{C}\{\ell\}$, respectively, where $\alpha_{\ell}$ is such that 
 \begin{equation}\label{eq:alpha_ell}
 	\alpha_\ell=\|D_{{A}\{\ell\}}^{-1}A\{\ell\}\|_2^{-1}.
 \end{equation}
 
For the sake of simplicity, we prove the level independence of the multigrid procedure described above involving the matrices $\hat{\mathcal{A}}_{2n_{\ell}}\{\ell\}$ defined by  formulas  (\ref{eq:Astep0})-(\ref{eq:recursive_AL_elements}) in the case of symbols vanishing in $\theta_0=0$, which arises from the discretization of boundary values problems. For the more general case $\theta_0 \neq 0$ we can apply a change of variable shifting the zero like in Remark 7 in \cite{AD}.

\begin{Lemma}\label{lem:Notay_symbol_wcycle}
Consider the matrices $\mathcal{A}_{2n_{\ell}}\{\ell\}\in \mathbb{C}^{2n_{\ell}\times 2n_{\ell}}$ defined by formulae (\ref{eq:Astep0})-(\ref{eq:recursive_AL_elements}). 
Define for all $\ell$ the grid transfer operators 
\begin{equation*}
\mathcal{P}\{\ell\}=\begin{bmatrix}
P_{A{\{\ell\}}}& \\
& P_{{\hat{C}\{\ell\}}}
\end{bmatrix}, \qquad 
P_{A{\{\ell\}}}=\mathcal{C}_{n_\ell}(p_{A})K_{n_\ell}^T,
\qquad P_{\hat{C}\{\ell\}}=\mathcal{C}_{n_\ell}(p_{\hat{C}})K_{n_\ell}^T.
\end{equation*}
Suppose that $f_{A\{\ell\}}$, $f_{C\{\ell\}}$, $f_{B\{\ell\}}$, $\alpha_\ell$, $p_{A}$ and $p_{\hat{C}}$ fulfil the hypotheses of Theorem \ref{thm:Notay_symbol} with $\theta_0=0$. Moreover, assume that
\begin{equation*}
|p_{A}|^2(\theta)+ |p_{\hat{C}}|^2(\theta+\pi)>0,  \qquad\forall\,\theta\in [0,2\pi].
\end{equation*}
Then $f_{A\{\ell+1\}}(0)=f_{B\{\ell+1\}}(0)=0$, $f_{A\{\ell+1\}}(\theta)>0$ and $f_{B\{\ell+1\}}(\theta)\neq0$ for all $\theta\in (0,2\pi)$, and
\begin{equation}\label{eq:newlim}
\limsup_{\theta\rightarrow 0}\frac{f_{A\{\ell\}}(\theta)}{f_{A\{\ell+1\}}(\theta)}=c_1, \qquad
\limsup_{\theta\rightarrow 0}\frac{\left|f_{B\{\ell+1\}}(\theta)\right|^2}{\left|f_{B\{\ell\}}(\theta)\right|^2}=c_2, \qquad 
0<c_1,c_2<\infty.
\end{equation}
Moreover, $f_{{C}\{\ell+1\}}(\theta)\geq 0$ for all $\theta\in [0,2\pi]$ and
\begin{equation}\label{eq:newC}
\limsup_{\theta\rightarrow 0}\frac{f_{\hat{C}\{\ell\}}(\theta)}{f_{\hat{C}\{\ell+1\}}(\theta)}<\infty.
\end{equation}
\end{Lemma}

\begin{proof}
The {assertion} follows directly from Lemmas \ref{lem:Notay_symbol_wcycleb} and \ref{lem:Notay_symbol_wcycle_C} in \ref{sec:appendix}. 
\end{proof}

\begin{Theorem}\label{thm:Notay_symbol_wcycle}
Consider the matrices $\mathcal{A}_{2n_{\ell}}\{\ell\}\in \mathbb{C}^{2n_{\ell}\times 2n_{\ell}}$ defined by formulae (\ref{eq:Astep0})-(\ref{eq:recursive_AL_elements}). 
Define for all $\ell$ the grid transfer operators 
\begin{equation}\label{eq:Plteo}
\mathcal{P}\{\ell\}=\begin{bmatrix}
P_{A{\{\ell\}}}& \\
& P_{{\hat{C}\{\ell\}}}
\end{bmatrix}, \qquad 
P_{A{\{\ell\}}}=\mathcal{C}_{n_\ell}(p_{A})K_{n_\ell}^T,
\qquad P_{\hat{C}\{\ell\}}=\mathcal{C}_{n_\ell}(p_{\hat{C}})K_{n_\ell}^T.
\end{equation}
Suppose that $f_{A\{0\}}$, $f_{C\{0\}}$, $f_{B\{0\}}$, $\alpha_0$, $p_{A}$ and $p_{\hat{C}}$ fulfil the hypotheses of Theorem \ref{thm:Notay_symbol} with $\theta_0=0$. Moreover, assume that
\begin{equation*}
|p_{A}|^2(\theta)+ |p_{\hat{C}}|^2(\theta+\pi)>0,  \qquad\forall\theta\in [0,2\pi].
\end{equation*}
For each $\ell\geq0$, consider a TGM associated with one iteration of damped Jacobi as postsmoothing with relaxation parameter 
\begin{equation}\label{eq:choice_omega_wcycle}
\omega_\ell<2\min\left(2 \alpha_\ell-\frac{\alpha_\ell^2}{\hat{a}_0(f_{A\{\ell\}})}\|f_{A\{\ell\}}\|_\infty\,,\, \hat{a}_0(f_{\hat{C}\{\ell\}})\left\|f_{{C}\{\ell\}}+\frac{|f_{B\{\ell\}}|^2}{f_{{A}\{\ell\}}}\right\|_\infty^{-1} \right).
\end{equation}
Then, the TGM iteration matrix involving the matrices $\hat{\mathcal{A}}_{2n_{\ell}}\{\ell\}$ defined in (\ref{eq:hatAl}) is such that
\begin{equation*}
\rho\left(TGM\left(\hat{\mathcal{A}}_{2n_{\ell}}\{\ell\}, {\mathcal{P}}\{\ell\}, \omega_\ell\right)\right)<1.
\end{equation*}
\end{Theorem}
\begin{proof}
The algebra structure of circulant matrices and Lemma \ref{lem:f_coarse} imply that the matrices $A\{\ell\}$, $C\{\ell\}$ and $B\{\ell\}$ defined by formulae (\ref{eq:Astep0})-(\ref{eq:recursive_AL_elements}) are circulant matrices themselves. We prove by induction that $f_{A\{\ell\}}$, $f_{C\{\ell\}}$, $f_{B\{\ell\}}$, $\alpha_\ell$, $p_{A}$ and $p_{\hat{C}}$ fulfil the hypotheses of Theorem \ref{thm:Notay_symbol} for all $\ell$ with $\theta_0=0$. 

For $\ell=0$, the hypotheses of Theorem \ref{thm:Notay_symbol} are fulfilled by assumption.

Then, we suppose that the hypotheses of Theorem \ref{thm:Notay_symbol} are fulfilled for $\ell$ and we prove them for $\ell+1$. 
By Lemma \ref{lem:Notay_symbol_wcycle} we have 
\begin{itemize}
\item $f_{A\{\ell+1\}}(0)=0$, $f_{A\{\ell+1\}}(\theta)>0$ for all $\theta\in (0,2\pi)$;
\item $f_{{C}\{\ell+1\}}(\theta) \geq 0$ for all $\theta\in [0,2\pi]$;
\item $f_{B\{\ell+1\}}(0)=0$, $f_{B\{\ell+1\}}(\theta)\neq0$ for all $\theta\in (0,2\pi)$;
\item 
\begin{equation*}
\limsup_{\theta\rightarrow 0}\frac{\left|f_{B\{\ell+1\}}(\theta)\right|^2}{f_{A\{\ell+1\}}(\theta)}=
\limsup_{\theta\rightarrow 0}\frac{\left|f_{B\{\ell+1\}}(\theta)\right|^2}{\left|f_{B\{\ell\}}(\theta)\right|^2}
							 \frac{\left|f_{B\{\ell\}}(\theta)\right|^2}{f_{A\{\ell\}}(\theta)}
							 \frac{f_{A\{\ell\}}(\theta)}{f_{A\{\ell+1\}}(\theta)}<\infty,
\end{equation*}
combining the induction hypothesis with the two limits in \eqref{eq:newlim}.
\end{itemize}

Concerning assumption 2.\ of Theorem \ref{thm:Notay_symbol}, for $A\{\ell+1\}$ we use the induction hypothesis 
$\limsup_{\theta\rightarrow 0}\frac{\left|p_{A}(\theta+\pi)\right|^2}{f_{A\{\ell\}}(\theta)}<\infty
$ and  \eqref{eq:newlim} so that we can write
\begin{equation}\label{eq:limsup_pA_on_fA}
\limsup_{\theta\rightarrow 0}\frac{\left|p_{A}(\theta+\pi)\right|^2}{f_{A\{\ell+1\}}(\theta)}=
\limsup_{\theta\rightarrow 0}\frac{\left|p_{A}(\theta+\pi)\right|^2}{f_{A\{\ell\}}(\theta)}\frac{f_{A\{\ell\}}(\theta)}{f_{A\{\ell+1\}}(\theta)}<\infty.
\end{equation}

Concerning $\hat{C}\{\ell+1\}$, we write
\[\limsup_{\theta\to {0}} \frac{|p_{\hat{C}}|^2(\theta+\pi)}{f_{\hat{C}\{\ell+1\}}(\theta)}
=\limsup_{\theta\to {0}} \frac{|p_{\hat{C}}|^2(\theta+\pi)}{f_{\hat{C}\{\ell\}}(\theta)}
\frac{f_{\hat{C}\{\ell\}}(\theta)}{f_{\hat{C}\{\ell+1\}}(\theta)}
<\infty,\]
where, for bounding the two terms, we used the induction hypothesis and formula (\ref{eq:newC}), respectively.

In conclusion, we proved that $f_{A\{\ell\}}$, $f_{C\{\ell\}}$, $f_{B\{\ell\}}$, $\alpha_\ell$, $p_{A}$ and $p_{\hat{C}}$ fulfil the hypotheses of Theorem \ref{thm:Notay_symbol} for all $\ell$ with $\theta_0=0$. Hence, Theorem \ref{thm:Notay_symbol} guarantees that $\rho\left(TGM\left(\hat{\mathcal{A}}\{\ell\}, {\mathcal{P}}\{\ell\}, \omega_\ell\right)\right)<1$, where $\omega_\ell$ is chosen according to (\ref{eq:choice_omega_wcycle}), and the proof is complete.
\end{proof}

When the matrices $A, B$, and $C$ are banded, then the matrix-vector product with matrix $\mathcal{A}$ in \eqref{eq:matrix_saddle} has a computational cost linear in $n$. Therefore, we would like to preserve the band structure of each block at the coarser levels such that each iteration of the V-cycle has a computational cost proportional to $n$.
This property is a consequence of the following lemmas that state that, at every level of the multigrid procedure, the generating functions of each block are trigonometric polynomials of degree lower than a constant independent of $n$. 

\begin{Lemma}\cite{AD}\label{lem:coarse_degree_psi}
Let $g$ be a trigonometric polynomial of degree $z_g$. Define 
\begin{equation}\label{eq:psi_g}
\psi(g)(\theta)=\frac{1}{2}\left[ g\left(\frac{\theta}{2}\right)+g\left(\frac{\theta}{2}+\pi\right)\right].
\end{equation}
Then, $\psi(g)$ is a trigonometric polynomial of degree at most $\left\lfloor{\frac{z_g}{2}}\right\rfloor$.
\end{Lemma}

This lemma implies that, for the classical multigrid method, the bandwidth of the coefficient matrix at the coarser levels becomes equal to the double of the bandwidth of the grid transfer operator, even when the coefficient matrix at the finer level has a large bandwidth.

To clearly distinguish the bandwidth of the grid transfer operators with respect to the coefficient matrices, we denote by
$q_A = z_{P_A}$ 
and
$q_{\hat{C}}=z_{P_{\hat{C}}}$ 
the degrees of the trigonometric polynomials $P_A$ and $P_{\hat{C}}$, respectively.

\begin{Lemma}\cite[Proposition 2]{AD} \label{lem:coarse_degree_A}
Let ${A\{\ell\}}$ be defined in \eqref{eq:recursive_AL_elements}, with $P_{A{\{\ell\}}}=\mathcal{C}_{n_\ell}(p_{A})K_{n_\ell}^T$, where $f_{A\{0\}}$ and $p_A$ are trigonometric polynomials of degree $z_{A\{0\}}$ and $q_A$, respectively. 
Then $f_{A\{\ell\}}$ is a trigonometric polynomial of degree $z_{A\{\ell\}}$ such that:
\begin{enumerate}
	\item $z_{A\{\ell\}} \le \max(z_{A\{0\}}, 2q_A)$ for all $\ell$;
	\item $z_{A\{\ell\}} \le 2q_A$ for $\ell$ large enough.
\end{enumerate}
\end{Lemma}

A similar result holds for the matrix \eqref{eq:matrix_saddle} when $A$, $B$, and $C$ are band matrices. 

\begin{Lemma}\label{lem:coarse_degree_B}
Consider the matrices $\mathcal{A}\{\ell\}\in \mathbb{C}^{2n_{\ell}\times 2n_{\ell}}$ defined by formulas (\ref{eq:Astep0})-\eqref{eq:recursive_AL_elements}, with $P_{A{\{\ell\}}}$ and $P_{\hat{C}{\{\ell\}}}$
defined as in \eqref{eq:Plteo}.
Let $$q=\max\{q_A, q_{\hat{C}}\},$$ where $q_A$ and $q_{\hat{C}}$ are the polynomial degrees associated with $p_{A}$ and $p_{\hat{C}}$, respectively. 

Then, for $\ell$ large enough, it holds 
\begin{enumerate}
\item $z_{B\{\ell\}}\le \max (2z_{B\{0\}}, 4q ),$
\item $z_{C\{\ell\}}\le \max(4 z_{B\{0\}},6q,2z_{C\{0\}}).$
\end{enumerate}
\end{Lemma} 
\begin{proof}
To prove item 1.\ we consider the function $\psi$ defined in equation \eqref{eq:psi_g}.
Exploiting the structure of $B\{\ell+1\}$ in \eqref{eq:recursive_AL_elements}, we have that the associated generating function is 
\[f_{B\{\ell+1\}}=\psi(g_{B\{\ell\}}), \qquad g_{B\{\ell\}}(\theta)=\overline{p}_{\hat{C}}p_{A}f_{B\{\ell\}}(\theta)\left(1-\frac{\alpha_\ell}{\hat{a}_0(f_{A\{\ell\}})}f_{A\{\ell\}}(\theta)\right).\]
Therefore, by Lemma \ref{lem:coarse_degree_psi}, we have
\[
z_{B\{\ell+1\}}\le \left \lfloor{ \frac{2q+z_{B\{\ell\}}+ z_{A\{\ell\}}}{2}}\right \rfloor 
\]
and, for $\ell$ large enough, item 2. of Lemma \ref{lem:coarse_degree_A} and $q_A\leq q$ give
\begin{equation}\label{eq:Bl+1}
z_{B\{\ell+1\}}\le 2q+\left \lfloor{ \frac{z_{B\{\ell\}}}{2}}\right \rfloor.
\end{equation}

We can now prove item 1.\ by induction {over} $\ell$.
For $\ell=0$ is trivial. 
For the induction step, inserting the induction assumption $z_{B\{\ell\}}\le \max (2z_{B\{0\}}, 4q )$ in  \eqref{eq:Bl+1}, we have 
\[
z_{B\{\ell+1\}}\le 2q+\left \lfloor{ \frac{\max (2z_{B\{0\}}, 4q )}{2}}\right \rfloor.
\]
Distinguishing the two cases in the maximum:
\begin{itemize}
\item Case $\max (2z_{B\{0\}}, 4q) = 4q$ implies
$z_{B\{\ell+1\}}\le 4q =  \max (2z_{B\{0\}}, 4q )$.
\item Case $\max (2z_{B\{0\}}, 4q) = 2z_{B\{0\}}$ implies
$z_{B\{\ell+1\}}\le 2q+z_{B\{0\}}\le 2\max(2q,z_{B\{0\}}) $.
\end{itemize}
Therefore, the item 1. follows in both cases.

We now prove item 2. exploiting the structure of ${C}\{\ell+1\}$ in \eqref{eq:recursive_AL_elements}. 
Since $f_{C\{\ell+1\}}=\psi(|p_{\hat{C}}|^2f_{\hat{C}\{\ell\}})$, recalling the definition $\hat{C}\{\ell\}$ in \eqref{eq:recursive_AL_elements} 
and applying Lemma \ref{lem:coarse_degree_psi}, we have
\begin{equation}\label{eq:Arico-Donatelli_on_C}
z_{C\{\ell+1\}}< \left \lfloor \frac{2q_{\hat{C}}+\max(z_{C\{\ell\}},2z_{B\{\ell\}}+ z_{A\{\ell\}})}{2}\right \rfloor.
\end{equation}
Moreover, for $\ell$ large enough, from item 1 follows that  $z_{B\{\ell\}}\le \max (2z_{B\{0\}}, 4q )$  while
item 2. of Lemma \ref{lem:coarse_degree_A} implies $z_{A\{\ell\}}\le 2q_{A}\le 2q$. Therefore, inserting these two majorizations in \eqref{eq:Arico-Donatelli_on_C},  we have 
\begin{equation}\label{eq:zCl+1}
z_{C\{\ell+1\}}< q+ \left \lfloor \frac{\max(z_{C\{\ell\}},\max (4z_{B\{0\}}, 8q )+ 2q)}{2}\right \rfloor.
\end{equation}

We can now prove item 2. by induction on $\ell$.
For $\ell=0$ is trivial. 
For the induction step, inserting the induction assumption $z_{C\{\ell\}}\le \max(4 z_{B\{0\}},6q,2z_{C\{0\}})$ in \eqref{eq:zCl+1},
we have
\begin{align*}
z_{C\{\ell+1\}} & < q+ \left \lfloor \frac{\max \left(\max(4 z_{B\{0\}},6q,2z_{C\{0\}}), \,\max (4z_{B\{0\}}, 8q )+ 2q \right)}{2}\right \rfloor\\
&= q+ \max\left( z_{C\{0\}}, 2z_{B\{0\}} + q, 5q\right).
\end{align*}
Distinguishing the three cases in the maximum $\xi =  \max\left( z_{C\{0\}}, 2z_{B\{0\}} + q, 5q\right)$, we have:
\begin{itemize}
\item Case $\xi = z_{C\{0\}}$ implies
$z_{C\{\ell+1\}}\le q + z_{C\{0\}} \leq 2 \max (q, z_{C\{0\}})$.
\item Case $\xi = 2z_{B\{0\}} + q$ implies
$z_{C\{\ell+1\}}\le 2q+2z_{B\{0\}}\le 2\max(2q,2z_{B\{0\}}) $.
\item Case $\xi=5q$ implies $z_{C\{\ell+1\}}\le 6q$.
\end{itemize}
Therefore, combining the three cases, we obtain $z_{C\{\ell+1\}}\le \max(4 z_{B\{0\}},6q,2z_{C\{0\}})$ which is item 2.
\end{proof}

\section{Numerical results}\label{sec:numerical_MGM}

The present section is devoted to the numerical validation of the theoretical results presented in Sections \ref{sec:TGM_conv_circ} and \ref{sec:procedure}. In all the experiments, we use the standard stopping criterion  $\frac{\|r^{(k)}\|_2}{\|\mathbf{b}\|_2}<\epsilon$, where $r^{(k)}=\mathbf{b}-{\mathcal{LAU}}\mathbf{x}^{(k)}$ and $\epsilon=10^{-6}$. The true solution $\mathbf{x}$ of the linear system ${\mathcal{LAU}}\mathbf{x} = \mathbf{b}$ is a uniform sampling of $\sin(t)$ on $[0,\pi]$ and we consider the right-hand side $\mathbf{b}$ defined as $\mathbf{b} = {\mathcal{LAU}}\mathbf{x}$, which automatically ensures that the requirements of Remark \ref{rmk:right-hand} are satisfied. We take the null initial guess. All the tests are performed using MATLAB 2021a and the error equation at the coarsest level is solved with the MATLAB backslash function after the proper projection of the residual into the range of the coefficient matrix. 


Firstly, we investigate the numerical behavior of the TGM applied to the elasticity problem described in Subsection~\ref{sec:numerical_TGM} with $\rho=1/2$. In Table \ref{tab:TGM_circ_1D_compare} we test the efficiency of the TGM  using 1 step of damped Jacobi as post smoother with four different values of $\omega$. The results show that the number of iterations needed for reaching the tolerance $\epsilon$ remains constant or decreases as the matrix size increases, confirming the theoretical optimal convergence rate. 
Moreover, the value $\omega_{\rm opt}$ obtained by minimizing the spectral radius estimate in (\ref{eq:omega}) proves to be the one associated to the minimum number of iterations, when compared to the other values belonging to a uniform sampling in the admissible interval $(0,1)$.

\begin{table}\begin{center}
		\caption{ Two-Grid iterations with different values of the Jacobi relaxation parameter $\omega$ and stopping tolerance $\epsilon= 10^{-6}$. }\label{tab:TGM_circ_1D_compare}
		\begin{tabular}{c|c|c|c|c}		
		$N=2\cdot (2^t) $ & \multicolumn{4}{c}{$\#$ Iterations}\\
			\hline
			{$t$} & {$\omega=1/4$} & {$\omega=1/2$}  &{$\omega_{\rm opt}=55/96$} & {$\omega=3/4$} \\
			\hline
			 9 & 34 & 14 & 12 & 15\\
			10 & 33 & 14 & 12 & 15\\
			11 & 32 & 14 & 11 & 14\\
			12 & 30 & 13 & 11 & 14\\
			13 & 29 & 13 & 11 & 13\\
			14 & 28 & 12 & 10 & 13\\	
		\end{tabular}
		\end{center}
\end{table}

As a second experiment, we numerically validate the convergence results of Theorem \ref{thm:Notay_symbol_wcycle} applying the W-cycle strategy analyzed in Section \ref{sec:procedure}. 
In particular, the matrices at the finest level  $A_{n_0}\{0\}=\mathcal{C}_{n_0}(f_{A\{0\}})$, $C_{n_0}\{0\}=\mathcal{C}_{n_0}(f_{C\{0\}})$,
 $B_{n_0}\{0\}=\mathcal{C}_{n_0}(f_{B\{0\}})$ are the circulant matrices given in equations (\ref{eq:ABC_circ}), with $n_0$ equal to $n=2^t$. 
If we choose $p_{\hat{C}}(\theta) = p_A(\theta)=\sqrt{2}(1+\cos(\theta))$, the hypotheses the Theorem~\ref{thm:Notay_symbol_wcycle} are fulfilled as long as we make a proper choice of the relaxation parameters $\omega_\ell$. In this way, the thesis of Theorem~\ref{thm:Notay_symbol_wcycle} suggests the convergence and optimality of the W-cycle procedure.

Concerning the choice of $\omega_\ell$, we propose an adaptive strategy approximating  at each level $\ell$ the quantities
\begin{equation}\label{eq:estimate_omega_ell}
2 \alpha_\ell-\frac{\alpha_\ell^2}{\hat{a}_0(f_{A\{\ell\}})}\|f_{A\{\ell\}}\|_\infty\,,\quad \hat{a}_0(f_{\hat{C}\{\ell\}})\left\|f_{{C}\{\ell\}}+\frac{|f_{B\{\ell\}}|^2}{f_{{A}\{\ell\}}}\right\|_\infty^{-1}.
\end{equation}
In particular, we exploit the fact that for each $\ell$ we have $f_{A\{\ell\}}(\theta)=2-2\cos(\theta)$ and hence ${\hat{a}_0(f_{A\{\ell\}})}=2$ and, according to formula (\ref{eq:alpha_ell}), $\alpha_\ell=1/2$. The invariance of the generating functions $f_{A\{\ell\}}$ at coarser levels with the choice $p_A(\theta)={\sqrt{2}}(1+\cos(\theta))$ can be derived from Lemma \ref{lem:f_coarse} by direct computation.

For the second quantity in (\ref{eq:estimate_omega_ell}), we read the Fourier coefficients of $f_{{B}\{\ell\}}$ and $f_{{C}\{\ell\}}$ from the first rows and columns of ${{B}\{\ell\}}$ and ${{C}\{\ell\}}$ and, in particular, we compute the Fourier coefficient $ \hat{a}_0(f_{\hat{C}\{\ell\}})$. Moreover, we approximate the $\|f_{{C}\{\ell\}}+\frac{|f_{B\{\ell\}}|^2}{f_{{A}\{\ell\}}}\|_\infty$ choosing the maximum value among its uniform sampling over $[0,\pi]$ with step-size $h=1/100$.

Then, according to (\ref{eq:choice_omega_wcycle}), we set at each level 
\begin{equation}
\label{eq:omega_ell}
\omega_\ell=\min\left(2 \alpha_\ell-\frac{\alpha_\ell^2}{\hat{a}_0(f_{A\{\ell\}})}\|f_{A\{\ell\}}\|_\infty\,,\, \hat{a}_0(f_{\hat{C}\{\ell\}})\left\|f_{{C}\{\ell\}}+\frac{|f_{B\{\ell\}}|^2}{f_{{A}\{\ell\}}}\right\|_\infty^{-1} \right), 
\end{equation}
which is the middle point of the interval of admissible values, inspired by the fact that at the finest level the $\omega_{\rm opt}=55/96$ is close to the middle point of the interval $(0,1)$.

In Table \ref{tab:Vcycle_cir} we report the number of iterations needed by the W-cycle method for reaching the convergence with tolerance $\epsilon$, comparing the previous adaptive choice of the parameters $\omega_\ell$ with the fixed choice of $\omega_\ell=1/2$.  We observe that the adaptive choice of the relaxation parameter at each level permits to obtain a number of W-cycle iterations which remains constant, confirming the theoretical optimal convergence rate. For this example, the choice of the fixed $\omega_\ell=1/2$ for the W-cycle is also valid and it provides the same behavior in terms of iterations. Indeed, the presence of the positive definite mass term $C$ implies that the choice of $\omega_\ell$ depends only on the quantities involving $A\{\ell\}$ that do not change level by level. 
 
 \begin{table}\begin{center}
		\caption{W-cycle iterations with the adaptive choice of $\omega_\ell$ according to Theorem \ref{thm:Notay_symbol_wcycle} and with the fixed choice $\omega_\ell= 1/2$ for all $\ell$.}\label{tab:Vcycle_cir}
		\begin{tabular}{c|c|c}		
		$N=2\cdot (2^t) $ & \multicolumn{2}{c}{$\#$ Iterations}  \\
			\hline
			{$t$} & {adaptive $\omega_\ell$}& {$\omega_\ell=1/2$} \\
			\hline
			 9 & 14 & 14 \\
			10 & 14 & 14 \\
			11 & 14 & 14 \\
			12 & 13 & 13 \\
			13 & 13 & 13 \\
			14 & 12 & 12 \\	
	\end{tabular}
		\end{center}
\end{table}
An analogous behavior, although {a} higher number of iterations,  is obtained using $p_{\hat{C}}\equiv 1$ in the construction of $P_{\hat{C}}$, see column 5 of Table \ref{tab:diff_rho}. Indeed, as already mentioned in Subsection \ref{sec:numerical_TGM}, this choice of $p_{\hat{C}}$ is still acceptable for reasonable large value of $\rho$. Taking smaller values of $\rho$, for example $\rho=1/20,1/200$ corresponds to weaken the positive definiteness of the $C\{0\}$ part. Then, formula (\ref{eq:limit_1_over_fC})    highlights that in this case is crucial to choose $p_{\hat{C}}(\theta)={\sqrt{2}}(1+\cos(\theta))$ in order to numerically satisfy item 2 of Theorem \ref{thm:Notay_symbol}. This can be numerically confirmed comparing columns 2-4 with columns 5-7 of Table \ref{tab:diff_rho}, where we observe a dramatic increasing of the number of iteration needed to achieve the convergence of the W-cycle passing from $p_{\hat{C}}(\theta)={\sqrt{2}}(1+\cos(\theta))$ to  $p_{\hat{C}}(\theta)=1$.


 \begin{table}\begin{center}
		\caption{W-cycle iterations with the adaptive choice of $\omega_\ell$ according to Theorem \ref{thm:Notay_symbol_wcycle}  for different values of $\rho$ and two different projectors by $p_{\hat{C}}(\theta)={\sqrt{2}}(1+\cos(\theta))$ and  $p_{\hat{C}}(\theta)=1$. }\label{tab:diff_rho}
		\begin{tabular}{c|c|c|c|c|c|c}		
		$N=2\cdot (2^t) $ & \multicolumn{3}{c|}{$p_{\hat{C}}(\theta)={\sqrt{2}}(1+\cos(\theta))$} & \multicolumn{3}{c}{$p_{\hat{C}}(\theta)=1$} \\
			\hline
			{$t$} & {$\rho=1/2$}& {$\rho=1/20$} & {$\rho=1/200$}& {$\rho=1/2$} & {$\rho=1/20$}& {$\rho=1/200$} \\
			\hline
			 9 & 14 & 17 & 18 & 24& 105 & 817\\
			10 & 14 & 16 & 18 & 24 & 107 & 842\\
			11 & 14 & 16 & 18 & 24 & 107 & 866\\
			12 & 13 & 16 & 18 & 24 & 107& 890\\
			13 & 13 & 15 & 17 & 24 & 108& 912\\
			14 & 12 & 15 & 17 & 24 & 108 & 932\\	
	
	\end{tabular}
		\end{center}
\end{table}


\section{Saddle point matrices with Toeplitz blocks}\label{sect:toep}
In the present section, we discuss the applicability of our multigrid method. For completeness, we report the definition of Toeplitz matrix generated by a function.

\begin{Definition}\label{def-multilevel}
The Toeplitz matrix associated with $f\in L^1(-\pi,\pi)$ is the matrix of order $n$ given by
\begin{align*}
  T_n(f)=\sum_{|j|<n}\hat a_j(f) J_{n}^{(j)},
\end{align*}
where $\hat a_j(f)$ are the Fourier coefficients of $f$ defined in \eqref{eq:coeff_fourier} and
$J_n^{(j)}$ is the matrix of order $n$ whose $(i,k)$ entry equals~$1$ if $i-k=j$
and zero otherwise.
\end{Definition}

%

\subsection{Multigrid methods for Toeplitz matrices}
\label{ssec:mult_toepl}
The theoretical results of Sections \ref{sec:TGM_conv_circ} and \ref{sec:procedure} can be naturally extended to systems in saddle-point form where the matrices $A$, $B$, $C$ belong to matrix algebras different from the circulant one. Indeed, as stressed in {\cite{AD,ADS}}, it is sufficient to substitute the Fourier transform in (\ref{eq:Circulant_diagonalization}) with the proper unitary transform and to choose the corresponding grid points. In particular, we first assume that $A=\tau_n(f_A)$, $BB^H=\tau_n(|f_B|^2)$, $C=\tau_n(f_C)$ belong to the $\tau$ algebra \cite{MR1311435}. In this case the structure of the grid transfer operators has to slightly change with respect to (\ref{eq:def_cutting_matrix_Kn_circ}) in order to preserve the algebra structure at coarser levels and the dimension of the problem should be $2n$ with odd 
$$n=2^t-1.$$ Then, the $(2^{t-1}-1)\times (2^{t}-1)$  cutting matrix $K_{n}$ takes the form
\[
K_{n} = \left[\begin{array}{ccccccccc}
0 & 1 & 0 & & & & &\\
&   & 0 & 1 & 0 & & & & \\
&   &   &   & \ddots & \ddots & \ddots & & \\
&   &   &   &        &        & 0      & 1 & 0		
\end{array}\right].
\]
Consequently, the grid transfer operators are chosen of the form
$\tau_{n}(p)K_{n}^T$.
Note that if $f_A$, $|f_B|^2$, and $f_C$ are  trigonometric polynomials of  degree  at  most  one, the related Toeplitz matrices belong to the $\tau$-algebra.   If, for instance,  the degree $\delta$ of $p_A$ is  greater than 1, then $T_n(p_A)$ is a Toeplitz matrix which differs from the $\tau$-matrix $\tau_n(p_A)$ for a low rank correction of rank at most  $2(\delta-1)$. 
Therefore, the associated grid transfer operator
\begin{equation}\label{eq:def_projector_pnk_toepl}
P_{A}=T_{n}(p_A)K_{n}^T
\end{equation}
should be adapted whenever $\delta>1$ to preserve the Toeplitz structure at the coarser levels (see \cite{ADS}).

\subsection{Numerical results for the elasticity problem with Toeplitz blocks}
\label{ssec:Num_toepl}
In the present subsection, we consider the problem (\ref{eq:1D_elasticity}) with Dirichlet boundary conditions. In this case, the analogous finite difference discretization with stepsize $h=1/(n+1)$ leads to a linear system with coefficient matrix of the form
\begin{equation} \label{eq:spectrally_analyzed_matrix_toepl}
\mathcal{A} = \mathcal{D}^{(1)}	\begin{bmatrix}
		A & hB^T \\
		hB & -h^2 C
	\end{bmatrix}\mathcal{D}^{(1)} = 	\begin{bmatrix}
		A& B^T \\
		B & -C
	\end{bmatrix},
	\qquad
	\mathcal{D}^{(1)}  = \begin{bmatrix}
		I & O  \\
		O & \frac{1}{h}I \\
	\end{bmatrix},	
\end{equation}
where $A,$ $C$, and $B$ are the Toeplitz matrices
\begin{align}\label{eq:ABC_toepl}
	A = T_n(2-2\cos(\theta)), \quad B = T_n(1-{\rm e}^{\hat{\imath}\theta}), \quad C = T_n\left(\frac{1}{3} (2+\cos(\theta))\right).
\end{align}
 We construct the grid transfer operators according to the $\tau$ algebra requirements, as described in Subsection \ref{ssec:mult_toepl}.  Exploiting the analysis for the circulant case in Subsection \ref{sec:numerical_TGM}, we firstly construct a TGM for $\mathcal{LAU}$ in the Toeplitz setting with a basic scheme, using only one step of damped Jacobi post-smoothing with relaxation parameter  $\omega=55/96$.
 Table \ref{tab:TGM_toep_1D_tol} and Figure \ref{fig:plot_rate_14_toep_TGM} show that the linear convergence independent of the matrix size of the TGM is preserved also in the Toeplitz case.


\begin{table}[ht]
  \begin{minipage}[b]{0.34\linewidth}
    \centering
    		\caption{Toeplitz case: TGM iterations with $\omega= 55/96$. }\label{tab:TGM_toep_1D_tol}
    \begin{tabular}{c|c}		
		$t $ & {$\#$ Iterations}\\
			\hline
			9& 13 \\
			10& 12  \\
			11& 12 \\	
			12& 11 \\
			13& 11 \\	
			14& 11  \\		
	\end{tabular}
  \end{minipage}%
  \begin{minipage}[b]{0.05\linewidth}
  \end{minipage}
  \hfill
  \begin{minipage}[c]{0.55\linewidth}
    \centering
        \captionof{figure}{Convergence rate for TGM method corresponding to $t=14$ of Table \ref{tab:TGM_toep_1D_tol}.}
    \label{fig:plot_rate_14_toep_TGM}
    \includegraphics[width=80mm]{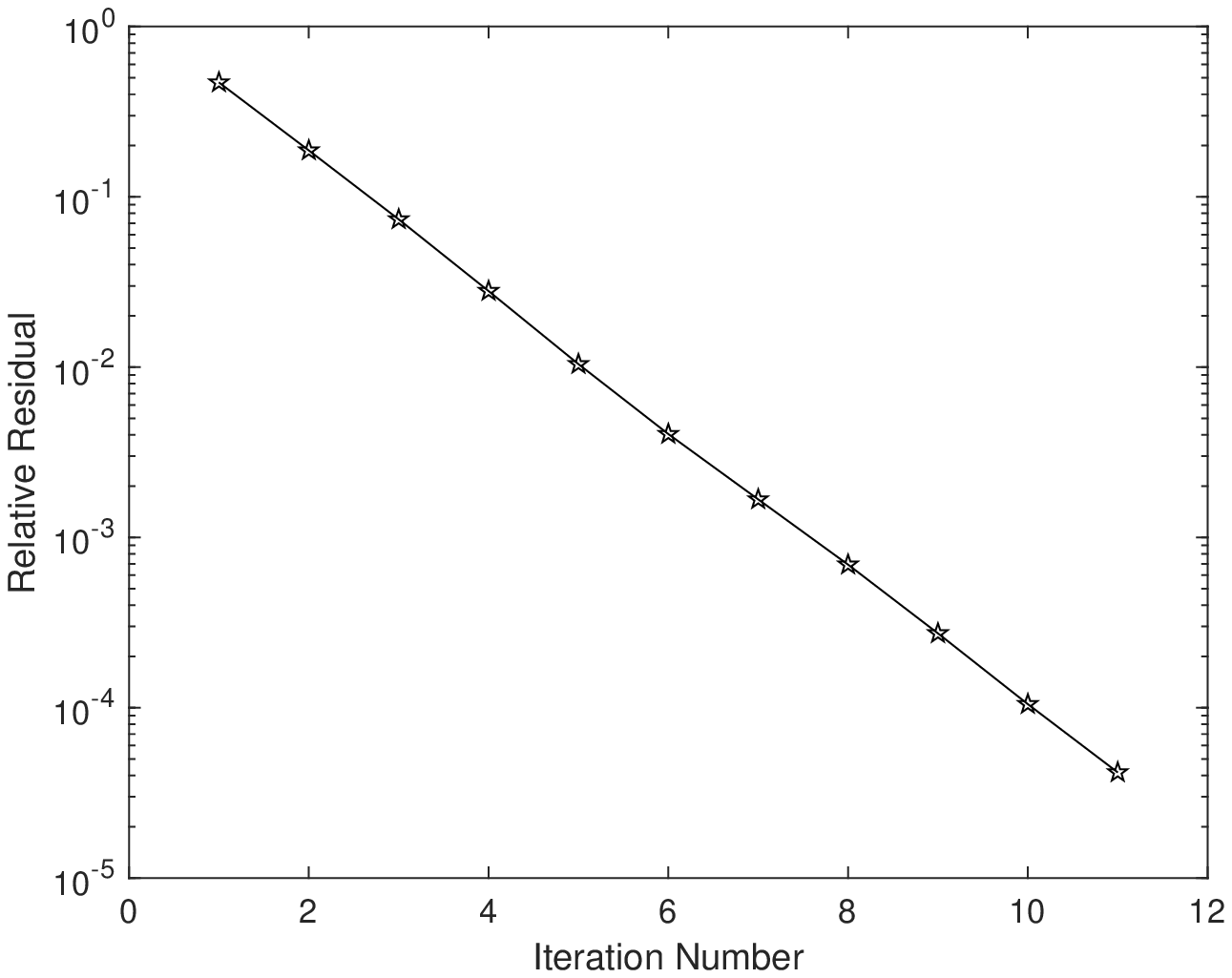}
  \end{minipage}
\end{table}


We conclude the section showing the efficiency of the proposed method in the Toeplitz case with more grids. For the definition of all the objects at coarser levels, we follow the recursive procedure described at the beginning of Section~\ref{sec:procedure}. Here the size of the matrix is $N=2\cdot (2^t-1)$. The grid transfer operators are associated to the trigonometric polynomials $p_{\hat{C}}(\theta) = p_A(\theta)={\sqrt{2}}(1+1\cos(\theta))$ as in Section \ref{sec:numerical_MGM}, while $P_{\hat{C}}$ and $P_A$ are defined according to Section \ref{ssec:mult_toepl}. Finally, the relaxation parameter $\omega_\ell$ at each level is set according to formula~(\ref{eq:omega_ell}).

The possible extension of Theorem \ref{thm:Notay_symbol_wcycle} to the $\tau$ algebra suggests that the convergence and optimality of the W-cycle method can be achieved also with Toeplitz blocks, even though in our case low-rank corrections are present.
Moreover, 
%
a linear convergence independent of the matrix size is obtained for both the W-cycle and V-cycle methods (see Table \ref{tab:vcycle_toep_1D_tol}) with a behavior similar to that of TGM, shown in Table~\ref{tab:TGM_toep_1D_tol}. The linear convergence is shown in Figure \ref{fig:plot_rate_14_toep_Wcycle_Vcycle}.


\begin{table}[ht]
  \begin{minipage}[b]{0.34\linewidth}
    \centering
		\caption{Toeplitz case: W-cycle and V-cycle iterations with the adaptive choice of $\omega_\ell$.}				
		\label{tab:vcycle_toep_1D_tol}
		\begin{tabular}{c|c|c}		
		$t$ & {W-cycle} & {V-cycle} \\
			\hline
			 9 & 13 &  14\\
			10 & 12 &  14\\
			11 & 12 & 14 \\
			12 & 11 & 13 \\
			13 & 11 & 13 \\
			14 & 11 & 13 \\		
	\end{tabular}
  \end{minipage}%
  \begin{minipage}[b]{0.05\linewidth}
  \end{minipage}
  \hfill
  \begin{minipage}[c]{0.55\linewidth}
    \centering
        \captionof{figure}{Convergence rate for W-cycle and V-cycle methods corresponding to $t=14$ of Table \ref{tab:vcycle_toep_1D_tol}.}
    \label{fig:plot_rate_14_toep_Wcycle_Vcycle}
    \includegraphics[width=80mm]{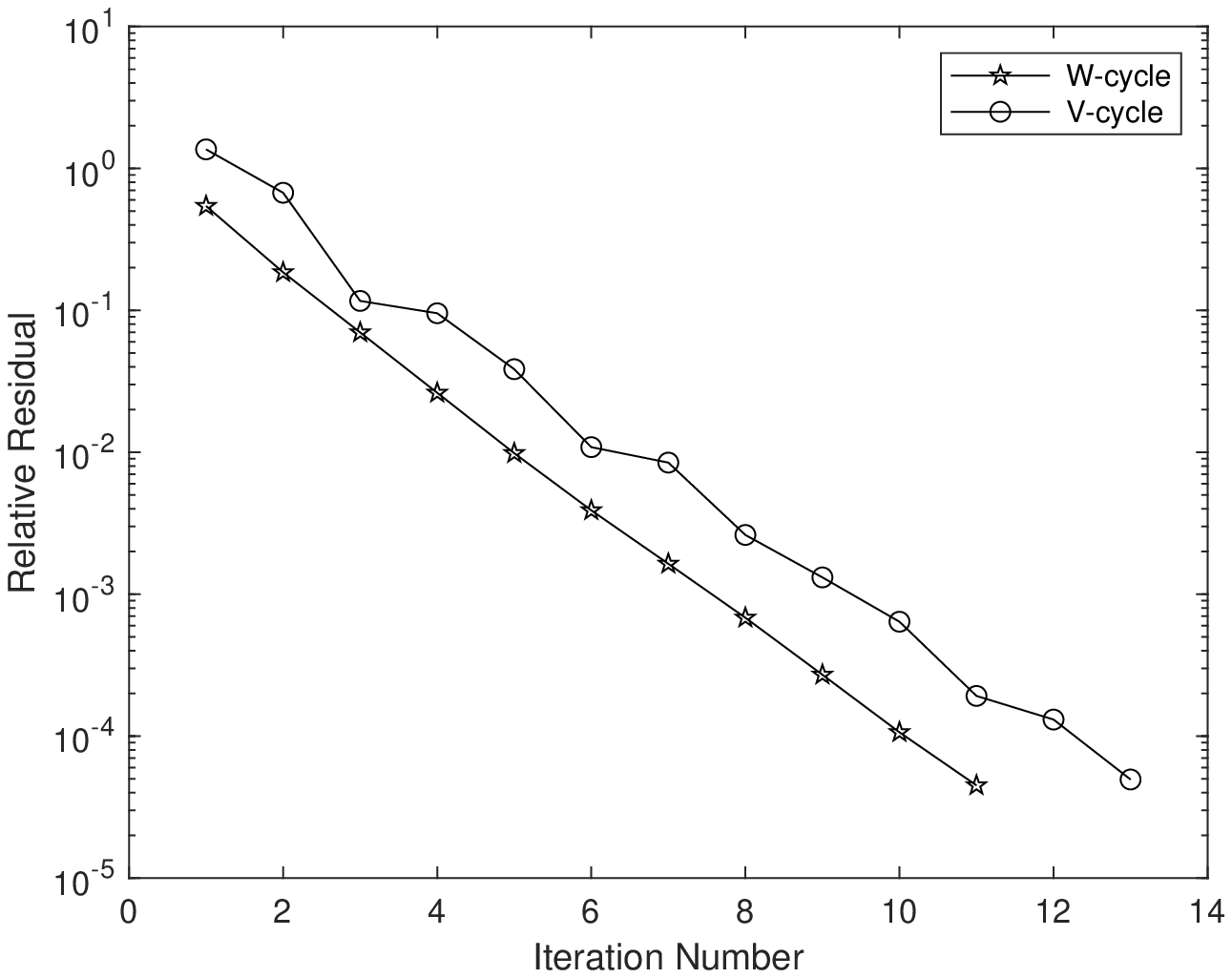}
  \end{minipage}
\end{table}

\section{Conclusions and future work}\label{sect:concl}
Using the results from \cite{MR3439215} we have provided sufficient conditions for the convergence of the TGM in the case of circulant blocks. Further, we have shown that the structure is kept on coarser levels and that the convergence rate is bounded independent from the level. While these are important results for the convergence of the W-cycle method, the convergence of the V-cycle remains open.

Using the analysis based on the generating symbols of the circulant blocks, we have been further able to provide optimal choices for the parameters $\alpha$ and $\omega$ in the left and right preconditioning and in the smoothing, respectively.

Numerically the resulting methods show the expected convergence behavior, demonstrating the validity of our analysis. Further, the W- and the V-cycle converge, as well.

In the future, we will extend the analysis to the multilevel case, where the generating symbols of $A, B$, and $C$ are multi-variate functions. This is of importance for applications that usually are posed in 2D or 3D, resulting in 2- or 3-level circulant or Toeplitz matrices. In this case we will have to consider different sizes of $A$ and $C$ and thus rectangular matrices $B$, as well.

\appendix
\section{Level Independency Proofs}\label{sec:appendix}
\begin{Lemma}\label{lem:Notay_symbol_wcycleb}
Consider the matrices $\mathcal{A}_{2n_{\ell}}\{\ell\}\in \mathbb{C}^{2n_{\ell}\times 2n_{\ell}}$ defined by formulae (\ref{eq:Astep0})-(\ref{eq:recursive_AL_elements}). 
Define for all $\ell$ the grid transfer operators 
\begin{equation*}
\mathcal{P}\{\ell\}=\begin{bmatrix}
P_{A{\{\ell\}}}& \\
& P_{{\hat{C}\{\ell\}}}
\end{bmatrix}, \qquad 
P_{A{\{\ell\}}}=\mathcal{C}_{n_\ell}(p_{A})K_{n_\ell}^T,
\qquad P_{\hat{C}\{\ell\}}=\mathcal{C}_{n_\ell}(p_{\hat{C}})K_{n_\ell}^T.
\end{equation*}
Suppose that $f_{A\{\ell\}}$, $f_{C\{\ell\}}$, $f_{B\{\ell\}}$, $\alpha_\ell$, $p_{A}$ and $p_{\hat{C}}$ fulfil the hypotheses of Theorem \ref{thm:Notay_symbol} with $\theta_0=0$. Moreover, assume that
\begin{equation*}
|p_{A}|^2(\theta)+ |p_{\hat{C}}|^2(\theta+\pi)>0,  \qquad\forall\theta\in [0,2\pi].
\end{equation*}
Then $f_{B\{\ell+1\}}(0)=0$, $f_{B\{\ell+1\}}(\theta)\neq0$ for all $\theta\in (0,2\pi)$ and
\begin{equation}\label{eq:limsup_fB}
\limsup_{\theta\rightarrow 0}\frac{\left|f_{B\{\ell+1\}}(\theta)\right|^2}{\left|f_{B\{\ell\}}(\theta)\right|^2}=c, \quad 0<c<\infty.
\end{equation}
\end{Lemma}

\begin{proof}
By Lemma \ref{lem:f_coarse}, for each $\ell$ we have 
\begin{equation*}
2f_{B\{\ell+1\}}(\theta)=\left(\overline{p_{\hat{C}}}f_{B\{\ell\}}\left(1-\frac{\alpha_\ell}{\hat{a}_0\left(f_{A\{\ell\}}\right)}f_{A\{\ell\}}\right)p_{A}\right)\left(\frac{\theta}{2}\right)+\left(
	\overline{p_{\hat{C}}}f_{B\{\ell\}}\left(1-\frac{\alpha_\ell}{\hat{a}_0\left(f_{A\{\ell\}}\right)}f_{A\{\ell\}}\right)p_{A}\right)\left(\frac{\theta}{2}+\pi\right).
\end{equation*}
Hence, the function $\left|f_{B\{\ell+1\}}(\theta)\right|^2$ can be written as the sum of three terms:
\begin{small}
\begin{equation}\label{eq:expression_of_fb}
\begin{split}
{\left|f_{B\{\ell+1\}}(\theta)\right|^2} &=\frac{1}{4}\left[ 
\left|f_{B\{\ell\}}\left(\frac{\theta}{2}\right)\right|^2\left|1-\frac{\alpha_\ell}{\hat{a}_0\left(f_{A\{\ell\}}\right)}f_{A\{\ell\}}\left(\frac{\theta}{2}\right)\right|^2\left|p_A\left(\frac{\theta}{2}\right)\right|^2\left|p_{\hat{C}}\left(\frac{\theta}{2}\right)\right|^2+\right.\\
&+ \left|f_{B\{\ell\}}\left(\frac{\theta}{2}+\pi\right)\right|^2\left|1-\frac{\alpha_\ell}{\hat{a}_0\left(f_{A\{\ell\}}\right)}f_{A\{\ell\}}\left(\frac{\theta}{2}+\pi\right)\right|^2\left|p_A\left(\frac{\theta}{2}+\pi\right)\right|^2\left|p_{\hat{C}}\left(\frac{\theta}{2}+\pi\right)\right|^2+\\
&+\left. 2\mathfrak{Re}\left(\overline{f_{B\{\ell\}}\left(\frac{\theta}{2}\right)}f_{B\{\ell\}}\left(\frac{\theta}{2}+\pi\right)
\left|1-\frac{\alpha_\ell}{\hat{a}_0\left(f_{A\{\ell\}}\right)}f_{A\{\ell\}}\left(\frac{\theta}{2}\right)\right|^2
\overline{p_{\hat{C}}\left(\frac{\theta}{2}\right)}p_{\hat{C}}\left(\frac{\theta}{2}+\pi\right)
\overline{p_A\left(\frac{\theta}{2}\right)}p_A\left(\frac{\theta}{2}+\pi\right) \right)\right].
\end{split}
\end{equation}
\end{small}

By assumption we have $f_{B\{\ell\}}(0)=0$, $f_{B\{\ell\}}(\theta)\neq0$ for all $\theta\in (0,2\pi)$, $|p_A|^2(\theta)+ |p_A|^2(\theta+\pi)>0$, $|p_{\hat{C}}|^2(\theta)+ |p_{\hat{C}}|^2(\theta+\pi)>0$ and $|p_A|^2(\theta)+ |p_{\hat{C}}|^2(\theta+\pi)>0$ for all $\,\theta\in [0,2\pi]$ and $p_A(\pi)=0$. These assumptions guarantee that for all  $\theta\in (0,2\pi)$ at least one term of the sum in (\ref{eq:expression_of_fb}) is different from 0. Hence, it is straightforward to check that $f_{B\{\ell+1\}}(0)=0$ and $f_{B\{\ell+1\}}(\theta)>0$ for all $\theta\in (0,2\pi)$.

If we prove that the limit involving each one of the three terms in (\ref{eq:expression_of_fb}) divided by $\left|f_{B\{\ell\}}(\theta)\right|^2$ is finite, then we can exploit the algebraic properties of limits and conclude that the $\limsup$ in (\ref{eq:limsup_fB}) is finite. We have to focus on the following quantities:
\begin{enumerate}
\item For the first term, the only quantity that should be checked is $\frac{\left|f_{B\{\ell\}}\left(\frac{\theta}{2}\right)\right|^2}{\left|f_{B\{\ell\}}(\theta)\right|^2}.$ Indeed, the other terms are proportional to a positive quantity $G_1(\theta)$ such that $0<c_1<G_1(\theta)<c_2<\infty$ in a neighborhood of 0. 
Consequently, we can write
 \[\limsup_{\theta\rightarrow 0}\frac{\left|f_{B\{\ell\}}\left(\frac{\theta}{2}\right)\right|^2}{\left|f_{B\{\ell\}}(\theta)\right|^2}G_1(\theta)
=G_1(0)<\infty.\]
\item For the second term, the quantity that should be checked is the ratio $\limsup_{\theta\rightarrow 0}\frac{\left|p_{A}\left(\frac{\theta}{2}+\pi\right)\right|^2}{\left|f_{B\{\ell\}}(\theta)\right|^2}$, since the other terms are proportional to a positive quantity $G_2(\theta)$ with $0<c_1<G_2(\theta)<c_2<\infty$ in a neighbourhood of 0. Consequently, we can write
\begin{equation}\label{eq:cond2W}
\limsup_{\theta\rightarrow 0}\frac{\left|p_{A}\left(\frac{\theta}{2}+\pi\right)\right|^2}{\left|f_{B\{\ell\}}(\theta)\right|^2}G_2(\theta)
=\limsup_{\theta\rightarrow 0}\frac{\left|p_{A}\left(\frac{\theta}{2}+\pi\right)\right|^2}{f_{A\{\ell\}}(\theta)}\frac{f_{A\{\ell\}}(\theta)}{\left|f_{B\{\ell\}}(\theta)\right|^2}G_2(\theta)
\end{equation}
and the finiteness of this $\limsup$ is implied by the hypotheses of Theorem \ref{thm:Notay_symbol}.
\item For the third term,
we bound it from above with its modulus and then it suffices to prove that 
\[
\limsup_{\theta\rightarrow 0}\frac{\left|f_{B\{\ell\}}\left(\frac{\theta}{2}\right)\right|}{\left|f_{B\{\ell\}}(\theta)\right|}\frac{\left|p_{A}(\frac{\theta}{2}+\pi)\right|}{\left|f_{B\{\ell\}}(\theta)\right|}G_3(\theta)=G_3(0)\limsup_{\theta\rightarrow 0}\frac{\left|p_{A}(\frac{\theta}{2}+\pi)\right|}{\left|f_{B\{\ell\}}(\theta)\right|}
\]
with $0<c_1<G_3(\theta)<c_2<\infty$ in a neighbourhood of 0.
The previous limit is finite thanks to \eqref{eq:cond2W}.
\end{enumerate}

We conclude the proof by highlighting that the limit in Item 1. is different from 0, which implies that the $\limsup$ in (\ref{eq:limsup_fB}) is different from 0.
\end{proof}

A further result can be proven for the symbols $f_{A\{\ell\}}$, $f_{{C}\{\ell\}}$ and $f_{\hat{C}\{\ell\}}$.

\begin{Lemma}\label{lem:Notay_symbol_wcycle_C}
Consider the matrices $\mathcal{A}_{2n_{\ell}}\{\ell\}\in \mathbb{C}^{2n_{\ell}\times 2n_{\ell}}$ defined by formulae (\ref{eq:Astep0})-(\ref{eq:recursive_AL_elements}). 
Define for all $\ell$ the grid transfer operators 
\begin{equation*}
\mathcal{P}\{\ell\}=\begin{bmatrix}
P_{A{\{\ell\}}}& \\
& P_{{\hat{C}\{\ell\}}}
\end{bmatrix}, \qquad 
P_{A{\{\ell\}}}=\mathcal{C}_{n_\ell}(p_{A})K_{n_\ell}^T,
\qquad P_{\hat{C}\{\ell\}}=\mathcal{C}_{n_\ell}(p_{\hat{C}})K_{n_\ell}^T.
\end{equation*}
Suppose that $f_{A\{\ell\}}$, $f_{C\{\ell\}}$, $f_{B\{\ell\}}$, $\alpha_\ell$, $p_{A}$ and $p_{\hat{C}}$ fulfil the hypotheses of Theorem \ref{thm:Notay_symbol} with $\theta_0=0$. Moreover, assume that
\begin{equation*}
|p_{A}|^2(\theta)+ |p_{\hat{C}}|^2(\theta+\pi)>0,  \qquad\forall\theta\in [0,2\pi].
\end{equation*}
Then $f_{A\{\ell+1\}}(0)=0$, $f_{A\{\ell+1\}}(\theta)>0$ for all $\theta\in (0,2\pi)$ and
\begin{equation}\label{eq:limsup_fA}
\limsup_{\theta\rightarrow 0}\frac{f_{A\{\ell\}}(\theta)}{f_{A\{\ell+1\}}(\theta)}=c, \quad 0<c<\infty.
\end{equation}
Moreover, the symbol $f_{{C}\{\ell+1\}}(\theta)$ associated with ${C}\{\ell+1\}$ is non-negative and
\begin{equation}\label{eq:limsup_fC}
\limsup_{\theta\rightarrow 0}\frac{f_{\hat{C}\{\ell\}}(\theta)}{f_{\hat{C}\{\ell+1\}}(\theta)}<\infty.
\end{equation}
\end{Lemma}
\begin{proof}
By Lemma \ref{lem:f_coarse}, we have 
\begin{equation*}
f_{A\{\ell+1\}}(\theta)=\frac{1}{2}\left(|p_A|^2f_{A\{\ell\}}\left(\frac{\theta}{2}\right)+
	|p_A|^2f_{A\{\ell\}}\left(\frac{\theta}{2}+\pi\right)\right).
\end{equation*}
Moreover, by assumption we have $f_{A\{\ell\}}(0)=0$, $f_{A\{\ell\}}(\theta)>0$ for all $\theta\in (0,2\pi)$, $|p_A|^2(\theta)+ |p_A|^2(\theta+\pi)>0$ for all $\,\theta\in [0,2\pi]$ and $p_A(\pi)=0$. Hence, as a consequence of Lemma \ref{lem:f_coarse}, we have $f_{A\{\ell+1\}}(0)=0$, $f_{A\{\ell+1\}}(\theta)>0$ for all $\theta\in (0,2\pi)$, and
\begin{equation*}
\limsup_{\theta\rightarrow 0}\frac{f_{A\{\ell\}}(\theta)}{f_{A\{\ell+1\}}(\theta)}=c, \quad 0<c<\infty.
\end{equation*}

Concerning $f_{\hat{C}\{\ell+1\}}$, we exploit the two definitions
\[
f_{\hat{C}\{\ell\}}(\theta)=f_{{C}\{\ell\}}(\theta)+ \frac{\alpha_{\ell}\left|f_{{B}\{\ell\}} \right|^2(\theta)}{\hat{a}_0\left(f_{{A}\{\ell\}}\right)}  \left(2
-\frac{\alpha_{\ell}}{\hat{a}_0\left(f_{{A}\{\ell\}}\right)}f_{{A}\{\ell\}}(\theta)\right)
=
f_{{C}\{\ell\}}(\theta)+ g_{1}(\theta)\left|f_{{B}\{\ell\}} (\theta)\right|^2
\]
and
\begin{equation}
\label{eq:expression_fhat_C_lplus1}
f_{\hat{C}\{\ell+1\}}(\theta)=f_{{C}\{\ell+1\}}(\theta)+ \frac{\alpha_{\ell+1}\left|f_{{B}\{\ell+1\}} \right|^2(\theta)}{\hat{a}_0\left(f_{{A}\{\ell+1\}}\right)}  \left(2
-\frac{\alpha_{\ell+1}}{\hat{a}_0\left(f_{{A}\{\ell+1\}}\right)}f_{{A}\{\ell+1\}}(\theta)\right)
=f_{{C}\{\ell+1\}}(\theta)+ g_{2}(\theta)\left|f_{{B}\{\ell+1\}} (\theta)\right|^2.
\end{equation}
The functions $g_1(\theta)$ and $g_2(\theta)$ are such that $0<c_1<g_1(\theta),g_2(\theta)<c_2<\infty$ in a neighborhood of 0 since $f_{{A}\{\ell\}}$ and $f_{{A}\{\ell+1\}}$ are non-negative and the $\alpha_{\ell}$ are chosen such that the function that maps $\theta$ into $2-\frac{\alpha_{\ell}}{\hat{a}_0\left(f_{{A}\{\ell\}}\right)}f_{{A}\{\ell\}}(\theta)$ is strictly positive for all $\ell$.

Note that $f_{{C}\{\ell\}}(\theta)\ge0$ by assumption, which implies that $f_{\hat{C}\{\ell\}}(\theta)$ is non-negative.
From Lemma \ref{lem:f_coarse} 
\begin{equation*}
f_{{C}\{\ell+1\}}(\theta)=\frac{1}{2}\left(|p_{\hat{C}}|^2f_{\hat{C}\{\ell\}}\left(\frac{\theta}{2}\right)+
|p_{\hat{C}}|^2f_{\hat{C}\{\ell\}}\left(\frac{\theta}{2}+\pi\right)\right)
\end{equation*}
which is a non-negative trigonometric polynomial.

In order to prove the bound in (\ref{eq:limsup_fC}) we exploit the expression of $\hat{C}\{\ell\}$ and $\hat{C}\{\ell+1\}$. In particular, 
\begin{align*}
\limsup_{\theta\rightarrow 0}\frac{f_{\hat{C}\{\ell\}}(\theta)}{f_{\hat{C}\{\ell+1\}}(\theta)}&=
\limsup_{\theta\rightarrow 0}\frac{f_{{C}\{\ell\}}(\theta)+ g_{1}(\theta)\left|f_{{B}\{\ell\}} (\theta)\right|^2}
{f_{{C}\{\ell+1\}}(\theta)+ g_{2}(\theta)\left|f_{{B}\{\ell+1\}} (\theta)\right|^2}\\
&=
\limsup_{\theta\rightarrow 0}\frac{f_{{C}\{\ell\}}(\theta)}
{f_{{C}\{\ell+1\}}(\theta)+ g_{2}(\theta)\left|f_{{B}\{\ell+1\}} (\theta)\right|^2}+
\limsup_{\theta\rightarrow 0}\frac{g_{1}(\theta)\left|f_{{B}\{\ell\}} (\theta)\right|^2}
{f_{{C}\{\ell+1\}}(\theta)+ g_{2}(\theta)\left|f_{{B}\{\ell+1\}} (\theta)\right|^2}.
\end{align*}
Now, we prove that both limits are finite. Concerning the first term, we have
\begin{align*}
\limsup_{\theta\rightarrow 0}\frac{f_{{C}\{\ell\}}(\theta)}{f_{{C}\{\ell+1\}}(\theta)+g_{2}(\theta)\left|f_{{B}\{\ell+1\}} (\theta)\right|^2}
&\le \limsup_{\theta\rightarrow 0}\frac{f_{{C}\{\ell\}}(\theta)}{f_{{C}\{\ell+1\}}(\theta)}\\
&\le \limsup_{\theta\rightarrow 0}\frac{2f_{{C}\{\ell\}}(\theta)}{|p_{\hat{C}}|^2f_{{C}\{\ell\}}\left(\frac{\theta}{2}\right)+
|p_{\hat{C}}|^2f_{{C}\{\ell\}}\left(\frac{\theta}{2}+\pi\right)+g_3(\theta)}<\infty
\end{align*}
where 
 $g_3(\theta)=|p_{\hat{C}}|^2g_1 |f_{{B}\{\ell\}}|^2\left(\frac{\theta}{2}\right)+|p_{\hat{C}}|^2g_1 |f_{{B}\{\ell\}}|^2\left(\frac{\theta}{2}+\pi\right)$  is a non-negative function and the latter bound is given by the result of Lemma \ref{lem:f_coarse}.
 
Finally, concerning the second term, the bound  
\begin{align*}
\limsup_{\theta\rightarrow 0}\frac{g_{1}(\theta)\left|f_{{B}\{\ell\}} (\theta)\right|^2}
{f_{{C}\{\ell+1\}}(\theta)+ g_{2}(\theta)\left|f_{{B}\{\ell+1\}} (\theta)\right|^2}
\le \limsup_{\theta\rightarrow 0}\frac{g_{1}(\theta)\left|f_{{B}\{\ell\}} (\theta)\right|^2}
{g_{2}(\theta)\left|f_{{B}\{\ell+1\}} (\theta)\right|^2}<\infty
\end{align*}
follows by the fact that $f_{\hat{C}\{\ell+1\}}$ is non-negative and by equation (\ref{eq:limsup_fB}) in Lemma \ref{lem:Notay_symbol_wcycle}.
\end{proof}

\section*{Acknowledgements}
The work of Marco Donatelli, Paola Ferrari, Isabella Furci is partially supported by Gruppo Nazionale per il Calcolo Scientifico (GNCS-INdAM). Moreover, the work of Isabella Furci was also supported by the Young Investigator Training Program 2020 (YITP 2019) promoted by ACRI.

\bibliographystyle{plain}
\bibliography{biblio}

\end{document}